\documentclass[12pt]{article}

\usepackage[utf8]{inputenc}
\usepackage[a4paper,nomarginpar]{geometry}
\usepackage[english]{babel}
\usepackage{enumitem}
\usepackage{amsmath,amsfonts,amssymb,amsthm,amscd}
\usepackage{tikz,xcolor}
\usepackage{mathrsfs}
\allowdisplaybreaks

\theoremstyle{plain}
\newtheorem{theorem}{Theorem}
\newtheorem{lemma}[theorem]{Lemma}
\newtheorem{corollary}[theorem]{Corollary}

\theoremstyle{definition}
\newtheorem{definition}[theorem]{Definition}
\newtheorem{remark}[theorem]{Remark}
\newtheorem{example}[theorem]{Example}

\newcommand{\N}{\mathbb{N}} 
\newcommand{\Z}{\mathbb{Z}} 
\newcommand{\R}{\mathbb{R}} 
\newcommand{\C}{\mathbb{C}} 
\newcommand{\D}{\mathbb{D}} 
\newcommand{\T}{\mathbb{T}} 
\newcommand{\K}{\mathbb{K}} 

\newcommand{\supp}{\operatorname{supp}}

\newcommand{\udens}{\operatorname{\overline{dens}}}

\newcommand{\eps}{\varepsilon}

\title{Shadowing for weighted composition operators on spaces of continuous functions}

\author{Jo\~ao V. A. Pinto}

\date{}

\begin{document}
	
	\maketitle
	
	\begin{abstract}
In this work, we characterize the shadowing property for weighted shifts on a general class of Banach sequence spaces, extending a theorem proved by Bernardes et al. in ETDS (2020). We then apply this result to characterize this property for weighted composition operators acting on a class of Banach spaces of continuous functions. As an application, we present a characterization of shadowing for bilateral weighted translations on the spaces $C_b(\R)$ and $C_0(\R)$. In a more general setting, we provide sufficient conditions for weighted composition operators to have the shadowing property. These conditions allow us to obtain complete characterizations of this property for multiplication operators on the Hardy spaces $H^p(\mathbb{D})$, $p\in[1,\infty]$, as well as on spaces of the form $H^\infty(\Omega)$ or $C_b(\Theta)$.
\end{abstract}

	
	\section{Introduction}
    This work lies in the area of linear dynamics, a branch of mathematics at the intersection of dynamical systems and operator theory. The field studies dynamical properties of continuous linear operators acting on topological vector spaces. Historically, it was primarily concerned with hypercyclicity (the existence of a dense orbit) and related dynamical properties, such as mixing and Devaney chaos. Much of the research conducted during the twentieth century and the first decade of the twenty-first century focused on these topics. For detailed treatments of these concepts, we refer the reader to \cite{BaMa,GE-P}.

    Over the past decade, other notions of chaos, such as distributional chaos and Li--Yorke chaos, have been extensively studied in linear dynamics. These concepts are based on the dynamical behavior of pairs of points. For seminal works on these properties in the linear setting, we refer the reader to \cite{Nilson0, BBMP2, BBP, BBPW}. More recently, the study of concepts from hyperbolic dynamics, such as hyperbolicity, shadowing, expansivity, and stability, has gained considerable attention in this setting. Although some classical results concerning these properties in the linear context had already been obtained in works such as \cite{EisenbergHedlund1970, Hedlund1971, Mazur2000}, a systematic study of these notions for operators on Banach spaces was first carried out in \cite{BCDMP2018}.

    Subsequently, several works in this direction appeared. In \cite{BerMe}, the authors solved a long-standing open problem, dating back more than fifty years, by exhibiting structurally stable operators on Banach spaces that are not hyperbolic. In the seminal paper \cite{CirGolPuj}, the concept of generalized hyperbolicity was introduced, which will play an important role throughout this work. In \cite{BernardesMessaoudi2020}, it was proven that any generalized hyperbolic operator on any Banach space is structurally stable. In \cite{BerPer}, the authors proved that the shadowing and finite shadowing properties coincide for invertible operators on Banach spaces, while showing that this is no longer true in the more general context of Fréchet spaces. Finally, in \cite{BCDFP}, the notions of generalized hyperbolicity, shadowing, stability, and expansivity were generalized to the setting of operators on locally convex spaces, and several corresponding results were established. 

    Another important result in \cite{BerMe} is the characterization of the shadowing property for weighted shifts on the sequence spaces $c_0(\Z)$ and $\ell^p(\Z)$, $1\leq p<\infty$. Characterizing dynamical properties in the context of weighted shifts on sequence spaces is one of the main goals in linear dynamics, since the structural simplicity of these operators often allows for explicit characterizations, thus providing concrete examples and counterexamples. Moreover, a complete understanding of a given property in the setting of weighted shifts may lead to characterizations of the same property for more general classes of operators. As examples, we mention the recent works \cite{Nilson2, Pinto}, which characterize chaotic properties for weighted shifts on Fréchet sequence spaces. An important open problem is to generalize the result mentioned at the beginning of this paragraph to the setting of Fréchet sequence spaces.

    In the first part of this work, namely Section 3, we take a first step toward solving the problem mentioned at the end of the previous paragraph. More precisely, under suitable assumptions, we establish in Theorem \ref{Teo1} a characterization of the shadowing property for weighted shifts on general Banach sequence spaces. The assumptions considered are conditions (C1) and (C2), introduced in Definition \ref{DCS}. Condition (C1) may be regarded as a natural structural assumption on the underlying space, whereas (C2) enables us to employ the implication
\[
\text{generalized hyperbolicity}\Longrightarrow \text{shadowing},\tag{$\bigtriangleup$}
\]
in the proof of the implication ($\Leftarrow$) of Theorem \ref{Teo1}. We observe that the implication ($\bigtriangleup$) was proved in the context of Banach spaces in \cite[Theorem A]{BCDMP2018}. Under additional assumptions (see \cite[Theorem 6]{BCDFP}), this implication remains valid in the setting of locally convex spaces. In essence, the strategy behind the proof of Theorem \ref{Teo1} is to generalize the arguments of \cite[Theorem 18]{BerMe}. We hope that this generalization will serve as a starting point for further extensions to the setting of Fréchet sequence spaces.

Weighted composition operators constitute one of the most important classes of operators in linear dynamics, as they provide a rich source of examples throughout the theory. In particular, the weighted shifts and the weighted translations belong to this class. Let $X$ be a space of functions. A weighted composition operator is an operator of the form
$$
C_{w,f}(\varphi) = (\varphi \circ f) w \quad \text{for every } \varphi \in X,
$$
where $f$ (called the {\it composition function}) and $w$ (called the {\it weight function}) are suitably chosen functions. There is a vast literature on such operators in several settings. In the analytic and differentiable context, we refer, for instance, to \cite{AJK, Blois, Prz}. In the measure-theoretic setting, we mention the works \cite{BerBonPin26, Nilson12, Nilson2, DAnDarMai21, DarPi, GG}. Finally, in the context of spaces of continuous functions, we refer to \cite{BerBonPin26, Nilson12, Nilson2}. We observe that, in \cite{BerBonPin26, Nilson12, Nilson2}, all characterizations of the dynamical properties under consideration are established in two settings: the spaces $L^p(\mu)$ ($1\le p < \infty$) and the spaces $C_0(\Omega)$. Moreover, in \cite{DAnDarMai21}, the authors characterize the shadowing property for composition operators acting on $L^p(\mu)$ ($1\le p < \infty$) under two assumptions, namely dissipativity and bounded distortion.

A natural question therefore arises: is it possible to obtain a version of the characterization established in \cite{DAnDarMai21} for the spaces $C_0(\Omega)$, or more generally for spaces of continuous functions? In the second and main part of this work, namely Section 4, we provide such a version in Theorem \ref{MainTheorem} for weighted composition operators acting on spaces of continuous functions satisfying certain assumptions. Among these assumptions, we highlight condition (AS2), which plays a role analogous to that of dissipativity in \cite{DAnDarMai21}. As corollaries, we obtain complete characterizations of shadowing for weighted translations on the spaces $C_0(\mathbb{R})$ and $C_b(\mathbb{R})$. Before proving Theorem \ref{MainTheorem}, we establish Theorem \ref{theo1}, which, in a more general setting than that of Theorem \ref{MainTheorem}, provides sufficient conditions for weighted composition operators to have the shadowing property. This result allows us to completely characterize this property for multiplication operators acting on the Hardy spaces $H^p(\mathbb{D})$, $p\in[1,\infty]$, on $H^\infty(\Omega)$, where $\Omega\subset\mathbb{C}$ is a relatively compact open set, and on spaces of the form $C_b(\Theta)$, where $\Theta$ is a relatively compact subspace of a Hausdorff space.

We recall that one of the main problems in linear dynamics is to determine whether the shadowing property and generalized hyperbolicity are equivalent in the setting of invertible operators on Banach spaces (\cite[Problem F]{BCDFP}). More precisely, the problem is to determine whether the converse of the implication $(\bigtriangleup)$ holds in the context of Banach spaces. In the characterizations obtained in \cite{BerMe} and \cite{DAnDarMai21}, these two concepts were shown to be equivalent. However, our results in Section 4 indicate that this equivalence may not hold in the general setting of Banach spaces. In contrast to \cite{BerMe, DAnDarMai21}, in Theorem \ref{theo1}, a condition weaker than generalized hyperbolicity was used to obtain the shadowing property. This may provide an approach to answering Problem F from \cite{BCDFP} in the negative.
	\section{Preliminaries}
    Throughout this text, $\mathbb{K}$ denotes either the field $\mathbb{R}$ of real numbers or the field $\mathbb{C}$ of complex numbers, $\mathbb{Z}$ denotes the ring of integers, $\mathbb{N}$ denotes the set of all positive integers, and $\mathbb{N}_0 = \mathbb{N} \cup \{0\}$. Given sets $A$ and $B$, a subset $C\subset A$, and a function $f:A\to B$, we denote by $f|_C$ the restriction of $f$ to $C$. Furthermore, for $b\in B$, the notation $f|_C\equiv b$ means that $f(c)=b$ for all $c\in C$. In this work, the symbol $\bigsqcup$ will be used to denote a disjoint union. By an ({\it invertible}) {\em operator} on a Banach space $Y$, we mean a continuous (invertible) linear map $T : Y \to Y$. Recall that the {\em operator norm} of such a map is the number $\|T\| = \sup\{\|Ty\| : \|y\| \leq 1\}$. We will use throughout this work the Spectral Radius Theorem, which ensures that the limit
\[
\lim_{n\to\infty}\|T^n\|^{\frac{1}{n}}
\]
exists and is finite for every operator \(T\) on a Banach space. Below, we list some results concerning limits of sequences that will be used throughout the paper:
\begin{itemize}
    \item (Submultiplicative Fekete's Lemma) If \((a_n)_{n\in\mathbb{N}}\) is a sequence of positive real numbers satisfying $a_{m+n}\leq a_m a_n$ for all \(m,n\in\mathbb{N}\), then the limit $\lim_{n\to\infty} a_n^{\frac{1}{n}}$ exists and is equal to \(\inf_{n\in\mathbb{N}} a_n^{\frac{1}{n}}\).
     \item Let $(a_n)_{n\in\mathbb N}$ be a sequence of positive real numbers. If $\lim_{n\to\infty}a_n^{1/n}>1$, then there exist constants $C>0$ and $t>1$ such that $a_n\geq Ct^n$ for every $n\in\mathbb N$.
     \item Let $(a_n)_{n\in\mathbb N}$ be a sequence of positive real numbers. If $\lim_{n\to\infty}a_n^{1/n}<1$, then there exist constants $C>0$ and $t\in(0,1)$ such that $a_n\leq Ct^n$ for every $n\in\mathbb N$.
\end{itemize}
Next, we present the concepts that will be used throughout this work.
    \subsection{Shadowing} 
The shadowing property provides a way to relate approximate trajectories to genuine orbits of a dynamical system. Since its origins in the study of hyperbolic dynamics \cite{Ano, Bow, Bow2}, it has played an important role in the qualitative analysis of dynamical systems. In the following, we present its definition.
\begin{definition}
    Let $X$ be a metric space with metric $d$ and let $f:X\to X$ be a homeomorphism. Given $\delta>0$, a $\delta$-\emph{pseudotrajectory} of $f$ is a sequence $(x_j)_{i<j<k}$ in $X$, where $-\infty\leq i<k\leq \infty$ and $k-i\geq 3$, such that $d(f(x_j),x_{j+1})\leq \delta$ for all $i<j<k-1$. We say that $f$ has the \emph{shadowing property} if for every $\varepsilon>0$, there exists $\delta>0$ such that every $\delta$-pseudotrajectory $(x_j)_{j\in\mathbb{Z}}$ of $f$ is $\varepsilon$-\emph{shadowed} by a real trajectory of $f$, that is, there exists $x\in X$ such that
\[
d(x_j,f^j(x))<\varepsilon,\quad \text{for all }j\in\mathbb{Z}.
\]
\end{definition}
    The notions of {\it finite shadowing} and {\it positive shadowing} are obtained from the previous definition by replacing $\mathbb{Z}$ with sets of the form $\{0,\ldots,k\}$, where $k\in\mathbb{N}$, and with $\mathbb{N}_0$, respectively. In both cases, invertibility of $f$ is not required. By \cite[Remark 3]{BerPer}, in the context of invertible operators on Fréchet spaces, we have 
    $$\text{Shadowing}\overset{(\nLeftarrow) }{\Longrightarrow} \text{Positive Shadowing}\overset{(\nLeftarrow )}{\Longrightarrow} \text{Finite Shadowing}.$$
    On the other hand, by \cite[Theorem 1]{BerPer}, in the context of invertible operators on Banach spaces, we have 
    $$\text{Shadowing}\Longleftrightarrow  \text{Positive Shadowing}\Longleftrightarrow  \text{Finite Shadowing}.$$
    
     Recall that an operator $T$ on a Banach space $X$ is said to be \emph{hyperbolic} if $\sigma(T)\cap\mathbb{T}=\varnothing$, where $\sigma(T)$ denotes the spectrum of $T$ and $\mathbb{T}:=\{\lambda\in\mathbb{C}:|\lambda|=1\}$ is the unit circle. Next, we present a definition, first formally stated in \cite{CirGolPuj}, that generalizes the classical notion of hyperbolicity and is closely related to the shadowing property in the context of linear dynamics. We note that the definition below was generalized to the setting of Fréchet spaces in \cite[Definition 1]{BCDFP}.
      \begin{definition}\cite[Definition 1]{CirGolPuj}\label{defgen}
Let $\left( X, \left\| \cdot \right\|  \right)$ be a Banach space. An invertible operator $T:X \to X$ is said to be \emph{generalized hyperbolic} if there exists a topological direct sum decomposition
\[
X=M\oplus N
\]
such that the following conditions hold:
\begin{enumerate}
\item[(GH1)] $T(M)\subset M$ and $T^{-1}(N)\subset N$;
\item[(GH2)] There exist $c>0$ and $t\in(0,1)$ such that
\[
\|T^n y\|\le ct^n\|y\|
\quad\text{and}\quad
\|T^{-n}z\|\le ct^n\|z\| \quad \text{for every } y\in M, z\in N \text{ and }n\in\mathbb N_0.
\]
\end{enumerate}
\end{definition}
In the setting of Banach spaces, we have the following result, which establishes that generalized hyperbolicity implies the shadowing property. In the context of locally convex spaces, generalized hyperbolicity requires additional assumptions in order to guarantee the shadowing property (see \cite[Theorem 6]{BCDFP}). 
\begin{lemma}{\rm \cite[Theorem A]{BCDMP2018}}\label{LL1} Let $X$ be a Banach space, and let $T: X \to X$ be an invertible operator. If $T$ is generalized hyperbolic, then $T$ has the shadowing property.
\end{lemma}
Let $(X,S)$ and $(Y,T)$ be two linear dynamical systems, that is, $S: X \to X$ and $T: Y \to Y$ are operators on the Banach spaces $X$ and $Y$, respectively. We say that $T$ is a \textit{linear factor} of $S$ if there exists a {\it linear factor map} $\Pi$, 
i.e., a surjective operator $\Pi : X \to Y$ such that $\Pi \circ S = T \circ \Pi.$ In \cite{DAnDarMai21}, the following result was given, which establishes that the shadowing property is preserved under linear factors.
\begin{lemma}{\rm \cite[Lemma 4.2.2]{DAnDarMai21}}\label{LL2}
    Let $(X,S)$ and $(Y,T)$ be linear dynamical systems where $T$ is a linear factor of $S$. If $S$ has the shadowing property, then so does $T$.
\end{lemma}

    \subsection{Admissible Banach sequence spaces}
    To establish one of the main results of this paper, Theorem \ref{MainTheorem}, we first require Theorem \ref{Teo1}. In the following, we present the preliminaries needed for this theorem.
    \begin{definition}
        A Banach space $X$ which is a vector subspace of the product space $\mathbb{K}^{\mathbb{Z}}$ is a {\it Banach sequence space over }$\Z$ if the inclusion map $X \to \mathbb{K}^{\mathbb{Z}}$ is continuous, i.e., convergence in $X$ implies coordinatewise convergence.
    \end{definition}
	Let $X$ be a Banach sequence space over $\Z$. If $w := (w_n)_{n\in\mathbb{Z}}$ is a sequence of nonzero scalars, the Closed Graph Theorem implies that the \emph{bilateral weighted backward shift} defined by
	\[
	B_w((x_n)_{n\in\mathbb{Z}}) := (w_n x_{n+1})_{n\in\mathbb{Z}},  
	\]
	
	is an operator on $X$ provided that it maps $X$ into itself. In the invertible case, we have 
    $$B_w^{-1}((x_n)_{n \in \mathbb{Z}}) = \left( \frac{1}{w_{n-1}} x_{n-1} \right)_{n \in \mathbb{Z}} \quad \text{for all }(x_n)_{n \in \mathbb{Z}}\in X.$$
    Let $X$ be a Banach sequence space over $\Z$. Consider the following condition:
    \begin{itemize}
        \item[(CB)] The canonical vectors $e_n := (\delta_{n,j})_{j\in\mathbb{Z}} \in \mathbb{K}^{\mathbb{Z}}$ belong to $X$ and $X$ can be written as the topological direct sum $X= M\oplus N$, where
        $$M=\left\{ (x_n)_{n\in \mathbb{Z}}\in X : x_n=0,\, \forall\, n>0 \right\}\quad \text{and}\quad N=\left\{ (x_n)_{n\in \mathbb{Z}}\in X : x_n=0, \, \forall\, n\leq 0 \right\}.$$
    \end{itemize}
    \begin{definition}\label{DCS}
        Let $X$ be a Banach sequence space over $\Z$ that satisfies condition (CB). Consider a bilateral weighted backward shift $B_w$, with nonzero weights $w := (w_n)_{n \in \mathbb{Z}}$, which is an invertible operator on $X$. We say that $X$ is an \textit{admissible Banach sequence space} {\it for} $B_w$ if the following conditions hold:
        \begin{itemize}
            \item[\rm (C1)] For each $m\in \Z$ and $x=(x_j)_{j \in \Z}\in X$, we have $\left| x_m \right|\left\| e_m \right\| \le \left\| x \right\|;$
            \item[\rm (C2)] For each $n\in \N$ and $x\in X$, we have 
            \begin{enumerate}
                 \item $\left\| (B_w)^n(x) \right\|\le \underset{k\in\Z, \,  x_k\ne 0}{\sup}|w_{k-n}\cdots w_{k-1}|\frac{\left\| e_{k-n} \right\|}{\left\| e_k \right\|}\left\| x \right\|$ and
                \item $\left\| (B^{-1}_w)^n(x) \right\|\le \underset{k\in\Z, \,  x_k\ne 0}{\sup}\frac{1}{|w_{k}\cdots w_{k+n-1}|}\frac{\left\| e_{k+n} \right\|}{\left\| e_k \right\|}\left\| x \right\|.$
              \end{enumerate}
        \end{itemize}
    \end{definition}
    \begin{lemma}\label{lnor}
         Let $X$ be an admissible Banach sequence space for the bilateral weighted backward shift $B_w$, with nonzero weights $w := (w_n)_{n \in \mathbb{Z}}$. Then the following assertions hold:
         \begin{itemize}
             \item[\rm (a)] $\left\| (B_w)^n \right\|= \underset{k\in\Z}{\sup}|w_{k-n}\cdots w_{k-1}|\frac{\left\| e_{k-n} \right\|}{\left\| e_k \right\|}$, for all $n\in \N$;
             \item[\rm (b)] $\left\| (B^{-1}_w)^n \right\|= \underset{k\in\Z}{\sup}\frac{1}{|w_{k}\cdots w_{k+n-1}|}\frac{\left\| e_{k+n} \right\|}{\left\| e_k \right\|}$, for all $n\in \N$;
              \item[\rm (c)] $\left\| (B_w|_M)^n \right\|= \sup_{k\in \N_0} \left| w_{-k-1}\cdots w_{-k-n} \right|\frac{\left\| e_{-k-n} \right\|}{\left\| e_{-k} \right\|}$ and $\left\| (B_w^{-1}|_N)^n \right\|= \sup_{k\in \N} \frac{1}{\left| w_{k}\cdots w_{k+n-1} \right|}\frac{\left\| e_{k+n} \right\|}{\left\| e_{k} \right\|}$ for all $n\in \N$, where $M$ and $N$ are the subspaces defined in condition (CB).
         \end{itemize}
    \end{lemma}
    \begin{proof}
        We prove only item (a), since the proofs of the remaining items are analogous. Fix $n\in \N$ and denote the supremum in item (a) by $s$. For each $k\in \Z$, we have 
        $$|w_{k-n}\cdots w_{k-1}|\frac{\left\| e_{k-n} \right\|}{\left\| e_{k} \right\|}= \left\| (B_w)^{n} \left( \frac{e_{k} }{\left\| e_{k}   \right\|} \right) \right\|\le \left\| (B_w)^{n} \right\|.$$
        Then, $s \le \left\| (B_w)^{n} \right\| $. On the other hand, condition (C2) ensures that $ \left\| (B_w)^{n} \right\|\le s$.
    \end{proof}
    \begin{example}\label{ex1}
        Let $\lambda=(\lambda_n)_{n\in\mathbb{Z}}$ be a sequence of nonzero scalars. The space $\ell^\infty(\lambda,\mathbb{Z})$ is defined by
\[
\ell^\infty(\lambda,\mathbb{Z}):=\left\{x=(x_n)_{n\in\mathbb{Z}}:\sup_{n\in\mathbb{Z}}|x_n\lambda_n|<\infty\right\},
\]
which is a Banach space when endowed with the norm $\|x\|_{\infty}:=\sup_{n\in\mathbb{Z}}|x_n\lambda_n|$, for $x \in \ell^\infty(\lambda,\mathbb{Z})$. Classical examples of closed subspaces of $\ell^\infty(\lambda,\mathbb{Z})$ are the spaces
\[
c_0(\lambda, \mathbb{Z}):=\left\{x=(x_n)_{n\in\mathbb{Z}}:\lim_{|n|\to\infty}x_n\lambda_n=0\right\}
\]
and
\[
c(\lambda,\mathbb{Z}):=\left\{x=(x_n)_{n\in\mathbb{Z}}:\exists\,L_1,L_2\in\K\text{ such that }\lim_{n\to-\infty}x_n\lambda_n=L_1\text{ and }\lim_{n\to\infty}x_n\lambda_n=L_2\right\}.
\]
For $1\leq p<\infty$, the space $\ell^p(\lambda,\mathbb{Z})$ is defined by
\[
\ell^p(\lambda, \mathbb{Z}):=\left\{x=(x_n)_{n\in\mathbb{Z}} :\sum_{n\in\mathbb{Z}}|x_n\lambda_n|^p <\infty\right\},
\]
which is a Banach space when endowed with the norm $\|x\|_p:=\left(\sum_{n\in\mathbb{Z}}|x_n\lambda_n|^p \right)^{1/p}$, for $x\in\ell^p(\lambda, \mathbb{Z})$. Now, consider nonzero weights $w:=(w_n)_{n\in \Z}$ such that 
\begin{equation}\label{cond}
\sup_{n\in\mathbb{Z}}\left|w_{n}\frac{\lambda_n}{\lambda_{n+1}}\right|
<\infty.
\end{equation}
Let $X\subset \ell^{p}(\lambda,\Z)$ ($1\le p \le \infty$) be a closed subspace such that
\begin{itemize}
    \item[(CS1)] $e_n\in X$ for all $n\in\Z$;
    \item[(CS2)] $B_w(X)= X$;
    \item[(CS3)] $X$ is invariant under the canonical projections $P_M$ and $P_N$, where $M$ and $N$ are the subspaces defined in condition (CB).
\end{itemize}
Then, using the Open Mapping Theorem, it is easy to verify that $X$ is an admissible Banach sequence space for $B_w$. Note that condition (CS3) is automatically satisfied for $1\le p < \infty$.

When $\lambda=(\lambda_n)_{n\in\mathbb{Z}}=(1)_{n\in \Z}$, we denote $\ell^p(\lambda,\mathbb{Z})$ ($1\leq p\le\infty$), $c_0(\lambda, \mathbb{Z})$ and $c(\lambda,\mathbb{Z})$ by $\ell^p(\mathbb{Z})$ ($1\leq p\le\infty$), $c_0(\mathbb{Z})$ and $c(\mathbb{Z})$, respectively.
    \end{example}
    \subsection{Admissible spaces of continuous functions}
    Let $\Omega$ be a locally compact Hausdorff space. We denote by $C_b(\Omega)$ the Banach space over $\mathbb{K}$ 
of all bounded continuous maps $\varphi : \Omega \to \mathbb{K}$ endowed with the supremum norm
$$ \|\varphi\|_{\infty} := \sup_{x \in \Omega} |\varphi(x)|.$$
We denote by $C_0(\Omega)$ the space of all continuous maps $\varphi : \Omega \to \mathbb{K}$ that vanish at infinity. Note that $C_0(\Omega)$ is a closed subspace of $C_b(\Omega)$. Given any bounded map $\varphi : \Omega \to \mathbb{K}$ and any $B \subset \Omega$, 
we define
\[
\|\varphi\|_{\infty,\,B} := \sup_{x \in B} |\varphi(x)|,
\]
where we consider this supremum to be $0$ if $B = \varnothing$. In the case $B = \Omega$, we usually write $\|\varphi\|_{\infty}$ instead of $\|\varphi\|_{\infty,\,\Omega}$. Recall that the \emph{support} of $\varphi : \Omega \to \mathbb{K}$ is defined by
\[
\operatorname{supp}\varphi :=\overline{ \{x \in \Omega : \varphi(x) \neq 0\}}.
\]
We also consider the space $C_c(\Omega)$, which is the subspace of $C_0(\Omega)$ consisting of those $\varphi \in C_0(\Omega)$
that have compact support. We now define the concepts that are particular to this work.
    \begin{definition}\label{DefIS}
        We say that the triple $(\Omega,f,w)$ is an {\it invertible system} if the following conditions hold:
        \begin{enumerate}
            \item $\Omega$ is a locally compact Hausdorff space;
            \item $w: \Omega \to \K$ is a continuous function such that $0 < \inf_{x \in \Omega} |w(x)| \le \sup_{x \in \Omega}|w(x)|< \infty$;
            \item $f: \Omega \to \Omega$ is a homeomorphism.
        \end{enumerate}
    \end{definition}
    Fix an invertible system $(\Omega,f,w)$. To simplify notation, we define the following sequence of functions $(w^{[n]})_{n\in \Z}$, where 
   $$ \begin{cases}
  w^{[0]}:=1, w^{[1]}:=w, w^{[-1]}:=\frac{1}{w\circ f^{-1}};  \\
w^{[n]}:=(w\circ f^{n-1})w^{[n-1]}, \quad \text{for } n\ge 2;  \\
w^{[-n]}:=\frac{1}{w\circ f^{-n}}w^{[-(n-1)]}, \quad \text{for } n\ge 2.
\end{cases}$$
\begin{definition}\label{DEF}
    Let $(\Omega,f,w)$ be an invertible system, and let $W\subset \Omega$ be a relatively compact and connected set with non-empty interior such that $W\subset \overline{\operatorname{int}(W)}$. Let $X$ be a closed subspace of $C_b(\Omega)$. We say that $X$ is an {\it admissible space for} $C_{w,f}$, {\it generated by } $W$, if the following conditions hold:
    \begin{itemize}
            \item[(AS1)] The {\it weighted composition operator} $C_{w,f}:X \to X$, defined by $\varphi \mapsto (\varphi\circ f)w$ for all $\varphi \in X$, is an invertible operator;
            \item[(AS2)] $\Omega = \bigsqcup_{k\in \Z} f^k(W)$;
            \item[(AS3)] There exist continuous functions $\rho_-,\rho_+: \Omega \to \K$ such that $\rho_{\pm} \varphi \in X$ for all $\varphi\in X$, $\left\| \rho_{\pm} \right\|_{\infty}=1$, $\operatorname{supp}(\rho_-)= \Omega^{0}_{-}$, $\operatorname{supp}(\rho_+)= \Omega^{0}_{+}$  and $\rho_- + \rho_+\equiv 1$, where 
            $$\Omega^{n}_{+}:=\overline{\bigsqcup_{k=n}^{\infty} f^k(W)}\quad \text{and} \quad \Omega^{n}_{-}:=\overline{\bigsqcup_{k=n}^{\infty} f^{-k}(W)}\quad \text{for all }n\in \N_0;$$
            \item[(AS4)] For each $x\in \Omega$, the set $S_x:=\left\{ (\varphi(f^n(x)))_{n\in\Z}: \varphi \in X \right\}$ is closed in $\ell^{\infty}(\Z)$ and satisfies conditions (CS1) and (CS3); 
            \item[(AS5)] $C_c(\Omega)\subset X$
        \end{itemize}
\end{definition}
 \begin{remark}
The purpose of the conditions in Definition \ref{DEF} is, similarly to what was done in \cite{DAnDarMai21} in a measure-theoretic setting, to use the characterization of the shadowing property for weighted shifts to characterize this property in the more general context of weighted composition operators acting on spaces of continuous functions. Conditions (AS2) and (AS4) play an analogous role to that of the dissipativity condition in \cite{DAnDarMai21}, while the relative compactness and connectedness assumptions on $W$, together with the uniform continuity of $w$ imposed in Theorem \ref{MainTheorem}, play an analogous role to that of the bounded distortion condition in \cite{DAnDarMai21}. Finally, condition (AS3) will allow us to use an argument weaker than generalized hyperbolicity in order to obtain the shadowing property in the implication ($\Leftarrow$) of Theorem \ref{MainTheorem}.
\end{remark}
\begin{remark}\label{rmkk0}
    In the context of the last definition, we make the following observations.
    \begin{itemize}
        \item[(a)] When $w\equiv 1$, we denote $C_{w,f}$ by $C_{f}$ and say that $C_f$ is a {\it composition operator}.
        \item[(b)] It is easy to see that $(C_{w,f})^n(\varphi)=(\varphi\circ f^n)w^{[n]}$ for all $n\in \Z$, and for all $\varphi\in X$. 
        \item[(c)] By (AS2), it follows that $f^n(x)\neq f^m(x)$ whenever $n\neq m$ for every $x\in\Omega$.
    \end{itemize}
\end{remark}
In the following, we present some lemmas that will be used in the proof of Theorem \ref{MainTheorem}.
    \begin{lemma}\label{lead1}
        Let $X$ be an admissible space for the weighted composition operator $C_{w,f}$. Then the following assertions hold: 
        \begin{itemize}
            \item[\rm (a)] For each $x\in \Omega$, $S_x$ is an admissible Banach sequence space for the bilateral weighted backward shift $B_{w_x}$ with weights $w_x:=(w(f^n(x)))_{n\in\Z}$.
            \item[\rm (b)] For each $x\in \Omega$, the limits $\lim_{n \to \infty} \sup_{k \in \mathbb{Z}} \left| w^{[n]}(f^k(x)) \right|^{\frac{1}{n}}$, $ \lim_{n \to \infty} \sup_{k \in \mathbb{Z}} \left| w^{[-n]}(f^k(x)) \right|^{\frac{1}{n}}$, \\
            $\lim_{n \to \infty} \sup_{k \in \mathbb{N}_0}  
    \left| w^{[n]}(f^{-k-n}(x)) \right|^{\frac{1}{n}}$, and $\lim_{n \to \infty} \sup_{k \in \mathbb{N}}  
    \left| w^{[-n]}(f^{k+n}(x)) \right|^{\frac{1}{n}}$ exist.
        \end{itemize}
    \end{lemma}
    \begin{proof}
        (a): Fix $x\in \Omega$. Since 
        $$B_{w_x}((\varphi(f^{n}(x)))_{n\in \mathbb{Z}})=(\varphi(f^{n+1}(x))w(f^n(x)))_{n\in \mathbb{Z}}= ((C_{w,f}(\varphi))(f^n(x)))_{n\in \mathbb{Z}}\in S_x,$$
         for all $\varphi \in X$, we have $B_{w_x}(S_x)\subset S_x$. Moreover, given $(\varphi(f^{n}(x)))_{n\in \mathbb{Z}}\in S_x$, for some $\varphi \in X$, we have $((C^{-1}_{w,f}(\varphi))(f^{n}(x)))_{n\in \mathbb{Z}}\in S_x$ and it is easy to see that
        $$B_{w_x}(((C^{-1}_{w,f}(\varphi))(f^{n}(x)))_{n\in \mathbb{Z}})=(\varphi(f^{n}(x)))_{n\in \mathbb{Z}}.$$
        Thus, $B_{w_x}(S_x)=S_x$. Furthermore, by Definition \ref{DefIS}, we have $\sup_{n \in \Z}|w(f^n(x))|< \infty$. Therefore, by condition (AS4) and Example \ref{ex1}, $S_x$ is an admissible Banach sequence space for $B_{w_x}$.

        (b): Fix $x\in \Omega$. By item (a), $S_x$ is an admissible Banach sequence space for $B_{w_x}$. Then, by Lemma \ref{lnor}, $\left\| (B_{w_x})^n \right\|_{\infty}= \sup_{k \in \mathbb{Z}} \left| w^{[n]}(f^k(x)) \right|$, $\left\| (B_{w_x}^{-1})^n \right\|_{\infty}=\sup_{k \in \mathbb{Z}} \left| w^{[-n]}(f^k(x)) \right|$,
        $\left\| (B_{w_x}|_M)^n \right\|_{\infty}=\sup_{k \in \mathbb{N}_0}\left| w^{[n]}(f^{-k-n}(x)) \right|$ and $\left\| (B_{w_x}^{-1}|_N)^n \right\|_{\infty}=\sup_{k \in \mathbb{N}} \left| w^{[-n]}(f^{k+n}(x)) \right|$, where the subspaces $M$ and $N$ come from Lemma \ref{lnor}. Therefore, by the Spectral Radius Theorem, we conclude item (b). 
    \end{proof}
    
    \begin{lemma}\label{lead2}
         Let $X$ be an admissible space for the weighted composition operator $C_{w,f}$. Then the limits $\lim_{n \to \infty} \left\| w^{[n]} \right\|_{\infty}^{\frac{1}{n}}$, $\lim_{n \to \infty} \left\| w^{[-n]} \right\|_{\infty}^{\frac{1}{n}}$, $\lim_{n \to \infty} \left\| w^{[n]} \right\|_{\infty,\,\Omega^{n}_{-}}^{\frac{1}{n}}$, and $\lim_{n \to \infty} \left\| w^{[-n]} \right\|_{\infty,\,\Omega^{n}_{+}}^{\frac{1}{n}}$ exist.         
    \end{lemma}
         \begin{proof}
        Using condition (AS5) and the same argument as in the proof of \cite[Theorem 15]{BerBonPin26}, we obtain that $\left\| (C_{w,f})^n \right\|_{\infty}=\left\| w^{[n]} \right\|_{\infty}$ for all $n\in \Z$. Then, by the Spectral Radius Theorem, the limits $\lim_{n \to \infty} \left\| w^{[n]} \right\|_{\infty}^{\frac{1}{n}}$ and $\lim_{n \to \infty} \left\| w^{[-n]} \right\|_{\infty}^{\frac{1}{n}}$ exist.
        
        Define $M:=\left\{ \varphi\in X: \varphi|_{\Omega_{+}^{1}}\equiv 0 \right\}$. Note that $M$ is closed and $C_{w,f}(M)\subset M$. Indeed, given $\varphi\in M$ and $x\in \Omega_{+}^{1}$, we have $f(x)\in \Omega_{+}^{2}\subset  \Omega_{+}^{1}$, and hence $\varphi (f(x))w(x)=0$. Therefore, $C_{w,f}(\varphi)\in M$. Now, fix $n\in \N$. Since 
        $$\left\| (C_{w,f}|_M)^n(\varphi) \right\|_{\infty}=\left\| (\varphi\circ f^n)w^{[n]} \right\|_{\infty}=\left\| (\varphi\circ f^n)w^{[n]} \right\|_{\infty,\,\Omega_{-}^{n}}\le \left\| \varphi \right\|_{\infty}\left\| w^{[n]} \right\|_{\infty,\,\Omega_{-}^{n}}$$
        for all $\varphi\in M$, we have $\left\| (C_{w,f}|_M)^n \right\|_{\infty}\le \left\| w^{[n]} \right\|_{\infty,\,\Omega_{-}^{n}}$. On the other hand, take $\eps > 0$. Since $W\subset \overline{\text{int}(W)}$, there exists $x \in \text{int}\left( \bigsqcup_{k=n}^{\infty}f^{-k}(W) \right)$ such that $|w^{[n]}(x)| \geq \|w^{[n]}\|_{\infty,\,\Omega_{-}^{n}} - \eps$. Choose an open neighborhood $V$ of $f^n(x)$ in $\text{int}\left( \bigsqcup_{k=0}^{\infty}f^{-k}(W) \right)$ such that $\overline{V}$ is compact. Then, there exists a continuous map $\phi : \Omega \to [0,1]$ such that 
$\supp \phi\subset V$ and $\phi(f^n(x)) = 1$. Hence, $\phi \in M$ and $\|\phi\|_{\infty} = 1$. Since
\[
\|(C_{w,f}|_M)^n(\phi)\|_{\infty} \geq |\phi(f^n(x)) w^{[n]}(x)| = |w^{[n]}(x)| \geq \|w^{[n]}\|_{\infty,\,\Omega_{-}^{n}} - \eps,
\]
we obtain $\|(C_{w,f}|_M)^n\|_{\infty} \geq \|w^{[n]}\|_{\infty,\,\Omega_{-}^{n}} - \eps$. Since $\varepsilon>0$ is arbitrary, we have $\left\| (C_{w,f}|_M)^n \right\|_{\infty}=\left\| w^{[n]} \right\|_{\infty,\,\Omega_{-}^{n}}$. Then, by the Spectral Radius Theorem, the limit $\lim_{n \to \infty} \left\| w^{[n]} \right\|_{\infty,\,\Omega^{n}_{-}}^{\frac{1}{n}}$ exists. 

Now, defining $N:=\left\{ \varphi\in X: \varphi|_{\Omega_{-}^{1}}\equiv 0 \right\}$ and using arguments analogous to the previous ones for $C_{w,f}^{-1}|_N$, one proves that the limit $\lim_{n \to \infty} \left\| w^{[-n]} \right\|_{\infty,\,\Omega^{n}_{+}}^{\frac{1}{n}}$ exists.
    \end{proof}
    The next lemma will enable us to exhibit concrete classes of admissible spaces for weighted composition operators.
    \begin{lemma}\label{lead3}
       Let $(\Omega,f,w)$ be an invertible system and consider $X:= C_0(\Omega)$ or $X:= C_b(\Omega)$. Let $C_{w,f}$ be a weighted composition operator that is an invertible operator on $X$. Suppose that the following conditions hold:
       \begin{itemize}
           \item[\rm (H1)] The conditions (AS2) and (AS3) hold, where $W\subset \Omega$, arising from (AS2), is a relatively compact and connected set with non-empty interior such that $W\subset \overline{\operatorname{int}(W)}$;
           \item[\rm (H2)] Each compact set $K \subset \Omega$ intersects $f^j(W)$ only for finitely many indices $j \in \Z$. 
       \end{itemize}
       Then $X$ is an admissible space for $C_{w,f}$, generated by $W$.
    \end{lemma}
    \begin{proof}
       We first claim that for each $x\in \Omega$, there exists a sequence $(B_n)_{n \in \Z}$ of pairwise disjoint sets, where $B_n$ is a compact neighborhood of $f^n(x)$ for each $n\in \Z$. Indeed, take $x \in \Omega$. Since $\Omega$ is locally compact, $x$ admits a compact neighborhood $V_0$. By condition (H2), there exist $k\in \N$ and $J:=\left\{ j_1,\cdots, j_k \right\}\subset \Z$, with $j_1=0$, such that $f^j(x)\in V_0$ only if $j \in J$. Using additionally the assumption that $\Omega$ is a Hausdorff space, we can get pairwise disjoint compact sets $(V_n)_{n=1}^{k}$ such that $V_n$ is a neighborhood of $f^{j_n}(x)$ for all $n=1,\cdots,k$. Hence, defining $V:=V_0 \cap V_1$, we have $V\cap \left\{ f^n(x):n\in\Z \right\}=\left\{ x \right\}$. Put $B_0:=V$. Suppose that, for some $n\in\mathbb{N}_0$, we have already constructed pairwise disjoint compact sets $(B_k)_{k=-n}^{n}$ such that $B_k$ is a neighborhood of $f^k(x)$ and 
        $$B_k\cap\{f^\ell(x):\ell\in\mathbb{Z}\}=\{f^k(x)\} \quad\text{for every } k=-n,\ldots,n.$$
        Since $\Omega$ is a locally compact Hausdorff space, it is, in particular, regular. Since $f^{n+1}(x)$ and $f^{-n-1}(x)$ do not belong to $\bigsqcup_{k=-n}^{n}B_k$ and are distinct, we can find disjoint compact neighborhoods $C_{n+1}$ and $C_{-n-1}$ of $f^{n+1}(x)$ and $f^{-n-1}(x)$, respectively, such that
\[
\left(C_{n+1}\cup C_{-n-1}\right)\cap\bigsqcup_{k=-n}^{n}B_k=\emptyset.
\]
        Define $B_{n+1}:=C_{n+1}\cap f^{n+1}(V)$ and $B_{-n-1}:=C_{-n-1}\cap f^{-n-1}(V)$. Therefore, by induction, we construct a family $(B_n)_{n\in\mathbb{Z}}$ of neighborhoods satisfying the required properties.

        $\underline{\text{The case }C_0(\Omega)}:$ it only remains to verify condition (AS4), since the remaining conditions follow from the hypotheses. Take $x\in \Omega$. We will prove that $S_x = c_0(\Z)$. Let $\varphi \in C_0(\Omega)$ and take $\varepsilon > 0$. Then there is a compact set $K \subset \Omega$ such that $|\varphi(y)|< \varepsilon$, for all $y \in \Omega \setminus K$. By condition (H2), there exists $n_0 \in \N$ such that 
$$K \subset \bigsqcup_{k=-n_0}^{n_0} f^k(W).$$
On the other hand, there exists $n_1\in \N$ such that for all $n \in \Z$ with $|n|> n_1$, we have $f^n(x) \notin \bigsqcup_{k=-n_0}^{n_0} f^k(W)$ and, therefore, $|\varphi(f^n(x))|< \varepsilon.$ Thus, $(\varphi(f^{n}(x)))_{n \in \Z} \in c_0$. Since $\varphi$ is arbitrary, it follows that $S_x \subset c_0(\Z)$. Now, take $(a_n)_{n\in\Z} \in c_0(\Z)$. By the first part of the proof, there exists a sequence $(B_n)_{n \in \Z}$ of disjoint sets, where $B_n$ is a compact neighborhood of $f^n(x)$. For each $n\in \Z$, we can define a function $\varphi_n \in C_0(\Omega)$, with $\varphi_n(f^{n}(x))=a_n$, $\left\| \varphi_n \right\|_{\infty}=|a_n|$ and $\text{supp}(\varphi_n) \subset B_n.$ Thus, 
$$\varphi: = \sum_{n \in \Z}^{}\varphi_n \in C_0(\Omega)\quad \text{and}\quad (a_n)_{n\in\Z}=(\varphi(f^n(x)))_{n\in\Z} \in S_x.$$
Therefore, $S_x = c_0(\Z)$ and hence condition (AS4) holds.

$\underline{\text{The case }C_b(\Omega)}:$ as before, it only remains to verify condition (AS4). Take $x\in \Omega$. We will prove that $S_x = \ell^{\infty}(\Z)$. By the first part of the proof, there exists a sequence $(B_n)_{n \in \Z}$ of disjoint sets, where $B_n$ is a compact neighborhood of $f^n(x)$. Moreover, since $B_n\subset f^n(V)$ for some compact set $V$ and for every $n\in \Z$, there exist $r\in \N$ and integers $i_1, \cdots, i_r$ such that $B_n \subset \bigsqcup_{k=1}^{r}f^{n+i_k}(W)$ for every $n\in\Z$. Therefore, for each $y\in \Omega$, there is an open neighborhood $V_y$ of $y$ such that
\begin{equation}\label{eqne2}
    V_y\cap B_n \neq \emptyset \quad \text{only for finitely many indices }n\in \Z.
\end{equation}
Take $(a_n)_{n\in\Z} \in \ell^{\infty}(\Z)$. For each $n\in \Z$, we can define a function $\varphi_n \in C_c(\Omega)$, with $\varphi_n(f^{n}(x))=a_n$, $\left\| \varphi_n \right\|_{\infty}=|a_n|$ and $\text{supp}(\varphi_n) \subset B_n.$ Note that $\varphi(y): = \sum_{n \in \Z}^{}\varphi_n(y)$ is well-defined for every $y\in \Omega$, and $\left\| \varphi \right\|_{\infty}=\sup_{n\in \Z}\left| a_n \right|<\infty$. Now, take $y\in \Omega$ and consider the neighborhood $V_y$ of $y$ from (\ref{eqne2}). Thus, there exist $s\in \N$ and integers $k_1, \cdots, k_s$ such that 
$$\varphi= \sum_{j=1}^{s}\varphi_{k_j}\quad \text{in }V_y.$$
Then $\varphi$ is continuous in $V_y$ and, in particular, in $y$. Since $y$ is arbitrary, we have that $\varphi$ is continuous. Hence, $\varphi \in C_{b}(\Omega)$ and $(a_n)_{n\in\Z}=(\varphi(f^n(x)))_{n\in\Z} \in S_x.$ Thus, $\ell^{\infty}(\Z)\subset S_x$. It is easy to see that $S_x \subset \ell^{\infty}(\Z).$ Therefore, $S_x=\ell^{\infty}(\Z)$ and hence condition (AS4) holds.
\end{proof}
\begin{example}\label{exx1}
Take $X:= C_0(\R)$ or $X:=C_b(\R)$. Consider the map $f:\R\to\R$ given by $f(x)=x+1$, and let $w: \R\to \R$ be a bounded continuous function. 
Then $C_{w,f}$ coincides with the {\it bilateral weighted translation operator} $T_w: X \to X$ given by 
$$T_w : \varphi \in X \mapsto \varphi(\cdot + 1) w(\cdot) \in X.$$
If $ \inf_{x\in \R} |w(x)|>0$, then, by Lemma \ref{lead3}, $X$ is an admissible space for $T_w$, generated by $W:=[0,1)$.  
\end{example}
     \section{The case of weighted shifts on admissible Banach sequence spaces}
     In this section, Theorem \ref{Teo1} provides a complete characterization of the shadowing property for weighted shifts on admissible Banach sequence spaces. This theorem generalizes \cite[Theorem 18]{BerMe} and may be regarded as a first advance toward obtaining a characterization of this property for weighted shifts on general Fréchet sequence spaces, which remains an important open problem in linear dynamics. Theorem \ref{Teo1} is one of the main ingredients in the proof of Theorem \ref{MainTheorem}, which is a central result of this paper. To establish it, we will need the following lemma, which is a version of \cite[Lemma 19]{BerMe}.
     \begin{lemma}\label{l1}
         Let $(w_n)_{n \in \mathbb{N}_0}$ be a sequence of nonzero scalars and let $(a_{n})_{n\in \N_0}$ be a sequence of positive scalars. Then we have the implications (b)$\Rightarrow$(a) and (c)$\Rightarrow$(a), where
\begin{enumerate}
\item[\rm (a)] $\lim_{n \to \infty} \sup_{k \in \mathbb{N}} \left| w_k w_{k+1} \cdots w_{k+n-1} \right|^{\frac{1}{n}}\frac{a_{k}^{\frac{1}{n}}}{ a_{k+n}^{\frac{1}{n}}} < 1;$
\item[\rm (b)] $\sup_{k \in \mathbb{N}} \sum_{n=1}^{\infty} \left| w_k w_{k+1} \cdots w_{k+n-1} \right|\frac{a_k}{a_{k+n}} < \infty;$

\item[\rm (c)] $\sup_{k \in \mathbb{N}} \sum_{n=1}^{k} \left|  w_{k-1} \cdots w_{k-n} \right|\frac{a_{k-n}}{a_{k}} < \infty.$
\end{enumerate}
     \end{lemma}
     \begin{proof}
         ((b)$\Rightarrow$(a)): For each $k\in \N$, define $R_k:=\sum_{n=1}^{\infty} \left| w_{k}  \cdots w_{k+n-1} \right|\frac{a_{k}}{a_{k+n}}$ and let $K>0$ be a real number such that $R_k<K$ for all $k \in\N$. Note that 
         $$\left| w_k \right|\frac{a_{k}}{a_{k+1}}=\frac{R_k}{1+R_{k+1}}$$
for each $k\in\N$. Thus, for a fixed $k\in \N$, we have
\begin{align*}
\left|  w_k w_{k+1} \cdots w_{k+n-1} \right|\frac{a_{k}}{a_{k+n}}&=|w_{k}|\frac{a_{k}}{a_{k+1}}|w_{k+1}|\frac{a_{k+1}}{a_{k+2}}\cdots\left|  w_{k+n-1}  \right|\frac{a_{k+n-1}}{a_{k+n}}\\
&= \frac{R_k}{1+R_{k+1}}\frac{R_{k+1}}{1+R_{k+2}}\cdots \frac{R_{k+n-1}}{1+R_{k+n}}\\
&=\frac{R_{k+1}}{1+R_{k+1}}\frac{R_{k+2}}{1+R_{k+2}}\cdots \frac{R_{k}}{1+R_{k+n}}\\
&\overset{ }{\le} \left( \frac{K}{K+1} \right)^{n-1}K
\end{align*}
for every $n\in\N$, where in the last inequality we used the fact that the function $x\mapsto \frac{x}{x+1}$ is strictly increasing for $x>0$. Therefore, by the Submultiplicative Fekete's Lemma, this implication holds.

((c)$\Rightarrow$(a)): One proves it analogously to the proof of (b)$\Rightarrow$(a) by defining $$R_k := \sum_{n=1}^{k} \left| w_{k-1} \cdots w_{k-n} \right| \frac{a_{k-n}}{a_k}$$ and observing that $\left| w_k \right|\frac{a_k}{a_{k+1}}(1+R_k)=R_{k+1}$ for each $k \in \mathbb{N}$.
     \end{proof}
   
    \begin{theorem}\label{Teo1}
        Let $X$ be an admissible Banach sequence space for the bilateral weighted backward shift $B_w$, with nonzero weights $w := (w_n)_{n \in \mathbb{Z}}$. Then $B_w$ has the shadowing property if and only if one of the following conditions holds: 
        \begin{itemize}
            \item[\rm (I)] $\lim_{n\to \infty}\sup_{k\in \Z}\left( \left| w_{k}\cdots w_{k+n-1} \right|\frac{\left\| e_{k} \right\|}{\left\| e_{k+n} \right\|} \right)^{\frac{1}{n}}<1;$
            \item[\rm (II)] $\lim_{n\to \infty}\inf_{k\in \Z}\left( \left| w_{k}\cdots w_{k+n-1} \right|\frac{\left\| e_{k} \right\|}{\left\| e_{k+n} \right\|} \right)^{\frac{1}{n}}>1;$
            \item[\rm (III)] $\lim_{n\to \infty}\sup_{k\in \N}\left( \left| w_{-k-1}\cdots w_{-k-n} \right|\frac{\left\| e_{-k-n} \right\|}{\left\| e_{-k} \right\|} \right)^{\frac{1}{n}}<1$ and \\ $\lim_{n\to \infty}\inf_{k\in \N}\left( \left| w_{k}\cdots w_{k+n-1} \right|\frac{\left\| e_{k} \right\|}{\left\| e_{k+n} \right\|} \right)^{\frac{1}{n}}>1$.
        \end{itemize}
    \end{theorem}
    \begin{proof}
        ($\Rightarrow$): Given $\varepsilon=1$, by the shadowing property, there exists $\delta>0$ such that every $\delta$-pseudotrajectory is $\varepsilon$-shadowed. First, we will prove that one of the two conditions below holds: 
        \begin{itemize}
            \item[\rm (i)] $\lim_{n\to \infty}\sup_{k\in \N}\left( \left| w_{k}\cdots w_{k+n-1} \right|\frac{\left\| e_{k} \right\|_{}}{\left\| e_{k+n} \right\|_{}} \right)^{\frac{1}{n}}<1;$
            \item[\rm (ii)] $\lim_{n\to \infty}\inf_{k\in \N}\left( \left| w_{k}\cdots w_{k+n-1} \right|\frac{\left\| e_{k} \right\|_{}}{\left\| e_{k+n} \right\|_{}} \right)^{\frac{1}{n}}>1.$ 
        \end{itemize}
        Suppose that condition (i) is false. Thus, by Lemma \ref{l1} applied with $(a_k)_{k\in \N_0}:=\left( \left\| e_k \right\| \right)_{k\in \N_0}$, there exist $s\in \N$ and $N_0\in \N$ such that 
        \begin{equation}\label{eq1}
            \sum_{n=1}^{N_0} \left| w_s w_{s+1} \cdots w_{s+n-1} \right|\frac{\left\| e_{s} \right\|}{\left\| e_{s+n} \right\|} > \frac{1 + \delta}{\delta^2}.
        \end{equation}
        Fix $N>N_0$ and consider the following sequence in $X$:
        $$x_0 := e^{i\theta_0} \frac{e_{s+N}}{\left\| e_{s+N} \right\|}, \quad 
x_k := B_w(x_{k-1}) + \delta e^{i\theta_k} \frac{e_{s+N-k}}{\left\| e_{s+N-k} \right\|}, \quad k = 1, \ldots, N,$$
        $x_n := B_w(x_{n-1})$ for all $n>N$, and $x_{-n}:=B_{w}^{-n}(x_0)$ for all $n\in \N$, where the scalars $\theta_0, \theta_1, \ldots, \theta_N$ are chosen so that $\theta_N=0$ and
        $$e^{i\theta_k} w_s w_{s+1} \cdots w_{s+N-k-1}
= \left| w_s w_{s+1} \cdots w_{s+N-k-1} \right|
\quad \text{for each } 0 \leq k \leq N-1.$$
It is easy to see that $(x_n)_{n\in \Z}$ is a $\delta$-pseudotrajectory. Then, by the shadowing property, there exists $y=(y_n)_{n\in \Z}\in X$ such that 
\begin{equation}\label{eq2}
    \left\| B_w^k(y)-x_k \right\|<1 \quad \text{for all } k\in \Z.
\end{equation}
Note that 
\begin{align*}
x_N & = B_w^N(x_0) + \delta \sum_{k=1}^N B_w^{\,N-k}\!\left( e^{i\theta_k}\,\frac{e_{s+N-k}}{\|e_{s+N-k}\|} \right)\\
&=\left(
\frac{|w_s \cdots w_{s+N-1}|}{\|e_{s+N}\|}
+ \delta  \left( \sum_{k=1}^{N-1} \frac{|w_s \cdots w_{s+N-k-1}|}{\|e_{s+N-k}\|} \right)+\frac{\delta }{\left\| e_s \right\|}
\right) e_s.
\end{align*}
Moreover, by (\ref{eq2}) applied with $k=0$ and condition (C1) applied in the $(s+N)^{\text{th}}$ coordinate, we have 
\begin{equation}\label{eq3}
    \left| y_{s+N}-\frac{e^{i\theta_0}}{\left\| e_{s+N} \right\|_{}} \right|\left\| e_{s+N} \right\|<1.
\end{equation}
Therefore, by (\ref{eq2}) applied with $k=N$ and condition (C1) applied in the $s^{\text{th}}$ coordinate, we have 
$$ \left| \frac{|w_s \cdots w_{s+N-1}|}{\|e_{s+N}\|}
+ \delta  \left( \sum_{k=1}^{N-1} \frac{|w_s \cdots w_{s+N-k-1}|}{\|e_{s+N-k}\|} \right)+\frac{\delta }{\left\| e_s \right\|} - w_s \cdots w_{s+N-1}y_{s+N} \right|\left\| e_s \right\|_{}<1.$$
By the triangle inequality and (\ref{eq3}) we have
\begin{equation}\label{eq44}
     \delta \sum_{k=1}^{N-1}|w_s \cdots w_{s+N-k-1}| \frac{\left\| e_s \right\|}{\|e_{s+N-k}\|} -\left| w_s \cdots w_{s+N-1} \right|\frac{\left\| e_s \right\|}{\left\| e_{s+N} \right\|}<1.
\end{equation}
By (\ref{eq1}) and (\ref{eq44}) we have that $\frac{\|e_{s+N}\|}
{\left| w_s \cdots w_{s+N-1} \right|\,\|e_s\|}< \delta$. Thus, multiplying the inequality (\ref{eq44}) by $\frac{\left\| e_{s+N} \right\|}{\delta\left| w_s \cdots w_{s+N-1} \right|\left\| e_s \right\|}$ we obtain
$$\sum_{k=1}^{N-1}
\frac{\|e_{s+N}\|}
{\left| w_{s+N-k}\cdots w_{s+N-1} \right|\,\|e_{s+N-k}\|}<\frac{\|e_{s+N}\|}
{\delta\left| w_s \cdots w_{s+N-1} \right|\,\|e_s\|}
+ \frac{1}{\delta}<1+\frac{1}{\delta}.$$
Since the last inequality is valid for all $N>N_0$, by Lemma \ref{l1} applied with the sequences $(w'_n)_{n\in \N_0}=\left( \frac{1}{|w_{s+n+1}|} \right)_{n\in \N_0}$ and $(a_n)_{n\in\N_0}=\left( \frac{1}{\left\| e_{s+n+1} \right\|} \right)_{n\in \N_0}$,  we conclude (ii). 

Now consider the following sequence space $\left( X_{0}, \left\| \cdot  \right\|_{0} \right)$ given by 
$$X_0:=\left\{ (x_n)_{n\in \Z}:(x_{-n})_{n\in \Z}\in X \right\}$$
and 
$$\left\| (x_n)_{n\in \Z} \right\|_0:=\left\| (x_{-n})_{n\in \Z} \right\|, \quad \text{for all } (x_n)_{n\in \Z} \in X_0.$$
Note that $X_{0}$ is a Banach sequence space, $(e_n)_{n\in \Z}\subset X_0$, and satisfies condition (C1). Furthermore, the map $\phi: X \to X_0$ defined by $(x_n)_{n\in \Z} \mapsto (x_{-n})_{n\in \Z}$ is an isometric isomorphism. Thus, the weighted backward shift $B_{w'}: X_0 \to X_0$ with weights $w'=(w'_n)_{n\in\Z}:=\left( \frac{1}{w_{-(n+1)}} \right)_{n\in\Z}$ is an invertible operator on $X_0$ since
\begin{equation}\label{eq11}
    B_{w'}=\phi\circ B_w^{-1}\circ \phi^{-1}.
\end{equation}
Since $B_w$ has the shadowing property, its inverse $B_w^{-1}$ also has the shadowing property. Moreover, since shadowing is preserved under conjugation by isometric isomorphisms, it follows from (\ref{eq11}) that $B_{w'}$ has the shadowing property. Therefore, repeating the previous arguments for $B_{w'}$, we obtain that one of the two conditions below holds:
\begin{itemize}
    \item[\rm (iii)] $\lim_{n \to \infty} \ \inf_{k \in \mathbb{N}} \left( \left|  w_{-k-1} \cdots w_{-k-n} \right|\frac{\left\| e_{-k-n} \right\|}{\left\| e_{-k} \right\|} \right)^{\frac{1}{n}} > 1;$
    \item[\rm (iv)] $\lim_{n \to \infty} \ \sup_{k \in \mathbb{N}} \left( \left|  w_{-k-1} \cdots w_{-k-n} \right|\frac{\left\| e_{-k-n} \right\|}{\left\| e_{-k} \right\|} \right)^{\frac{1}{n}} < 1.$
\end{itemize}
Thus, we have four possible cases: (i) and (iii), (i) and (iv), (ii) and (iii), or (ii) and (iv). Since (I) is equivalent to (i) and (iv), (II) is equivalent to (ii) and (iii), and (III) is equivalent to (ii) and (iv), it remains to show that (i) and (iii) cannot both hold. Indeed, suppose that both (i) and (iii) hold. Consider the following sequence in $X$:
$$z_0=\frac{e_0}{\left\| e_0 \right\|},\quad z_n=B_w(z_{n-1})+\delta \frac{e_0}{\left\| e_0 \right\|} \quad \text{and} \quad
z_{-n}=B_w^{-1}(z_{-(n-1)}+\delta \frac{e_0}{\left\| e_0 \right\|}) \text{ for every } n\in \N.$$
Then we have 
$$z_n=\delta\left( \frac{e_0}{\left\| e_0 \right\|}+ \sum_{k=1}^{n-1} w_{-1}\cdots w_{-k}\frac{e_{-k}}{\left\| e_0 \right\|}\right)+w_{-1}\cdots w_{-n}\frac{e_{-n}}{\left\| e_0 \right\|},$$
where $\sum_{k=1}^{0} w_{-1}\cdots w_{-k}\frac{e_{-k}}{\|e_0\|}=0$ by convention, and
$$z_{-n}= \frac{e_n}{\left\| e_0 \right\|w_0\cdots w_{n-1}}+\delta\sum_{k=1}^{n}\frac{e_k}{\left\| e_0 \right\|w_0\cdots w_{k-1}}$$
for each $n\in \N$. It is easy to see that $(z_n)_{n\in \Z}$ is a $\delta$-pseudotrajectory. Then, by the shadowing property, there exists $a:=(a_n)_{n\in \Z}\in X$ such that 
\begin{equation}\label{eq12}
    \left\| B_w^n(a)-z_n \right\|<1 \quad \text{for all } n\in \Z.
\end{equation}
Fix $k\ge 1$. Suppose that $a_{-k}\ne0$. By (iii), there exists $n\in\N$ such that
\begin{equation}\label{eq13}
    \left|  w_{-k-1} \cdots w_{-k-n} \right|\frac{\left\| e_{-k-n} \right\|}{\left\| e_{-k} \right\|}>\frac{1}{|a_{-k}|\left\| e_{-k} \right\|}.
\end{equation}
 Since the $(-n-k)^{\text{th}}$ coordinate of $(B_w)^n(a)$ is $w_{-k-1} \cdots w_{-k-n} a_{-k}$, by (\ref{eq12}) and condition (C1) applied to the $(-n-k)^{\text{th}}$ coordinate, we obtain $\left|  w_{-k-1} \cdots w_{-k-n} \right|\left\| e_{-k-n} \right\||a_{-k}|<1$, which contradicts (\ref{eq13}). Thus, $a_{-k}=0$ for all $k\in \N$. In a similar way, we obtain $a_{k}=0$ for all $k\in \N$. Then $a=a_0e_0$. Consider $n_0\in \N$ such that $\left|  w_{-k-1} \cdots w_{-k-n} \right|\frac{\left\| e_{-k-n} \right\|}{\left\| e_{-k} \right\|}>\frac{1}{\delta}$ for all $k\in \N_0$ and $n\ge n_0$. By (\ref{eq12}) applied to $n=n_0+1$ and condition (C1) applied to the $(-n_0)^{\text{th}}$ coordinate, we obtain $\left|  w_{-1} \cdots w_{-n_0} \right|\frac{\left\| e_{-n_0} \right\|}{\left\| e_{0} \right\|}<\frac{1}{\delta}$, which is a contradiction.

 ($\Leftarrow$): Suppose that (I) holds. Then there exist $C>0$ and $t\in (0,1)$ such that 
 \begin{equation}\label{eq14}
     \sup_{k\in \Z} \left| w_{k}\cdots w_{k+n-1} \right|\frac{\left\| e_{k} \right\|}{\left\| e_{k+n} \right\|} \le Ct^n\quad \text{for all }n\in \N. 
 \end{equation}
Then, by (\ref{eq14}) and condition (C2), we obtain
 $$\left\| B_w^n(x) \right\|\le \sup_{k\in \Z} \left| w_{k}\cdots w_{k+n-1} \right|\frac{\left\| e_{k} \right\|}{\left\| e_{k+n} \right\|}\left\| x \right\|\le Ct^n\left\| x \right\|$$
for all $x\in X$ and $n\in \N$. Thus, $B_w$ is generalized hyperbolic with topological direct sum decomposition $X= X\oplus \left\{ 0 \right\}$. Therefore, by Lemma \ref{LL1}, $B_w$ has the shadowing property. Assuming now that (II) holds, and using the same arguments as before, we obtain that $B_w$ is generalized hyperbolic with topological direct sum decomposition $X=\left\{ 0 \right\}\oplus X$ and, therefore, by Lemma \ref{LL1}, $B_w$ has the shadowing property. To finish, suppose that (III) holds. Consider the topological direct sum decomposition $X=M\oplus N$ from condition (CB). It is easy to see that $B_w(M)\subset M$ and $B_w^{-1}(N)\subset N$. By (III), there exist $C>0$ and $t\in (0,1)$ such that 
\begin{equation}\label{eq16}
    \sup_{k\in \N_0} \left| w_{-k-1}\cdots w_{-k-n} \right|\frac{\left\| e_{-k-n} \right\|}{\left\| e_{-k} \right\|} \le Ct^n \quad \text{and}\quad \sup_{k\in \N} \frac{1}{\left| w_{k}\cdots w_{k+n-1} \right|}\frac{\left\| e_{k+n} \right\|}{\left\| e_{k} \right\|} \le Ct^n
\end{equation}
for all $n\in\N$. Then, by (\ref{eq16}) and condition (C2), we obtain
 $$\left\| B_w^n(x) \right\|\le \sup_{k\in \N_0} \left| w_{-k-1}\cdots w_{-k-n} \right|\frac{\left\| e_{-k-n} \right\|}{\left\| e_{-k} \right\|}\left\| x \right\|\le Ct^n\left\| x \right\|$$
for all $x\in M$ and $n\in \N$, and 
 $$\left\| (B^{-1}_w)^n(x) \right\|\le\sup_{k\in \N} \frac{1}{\left| w_{k}\cdots w_{k+n-1} \right|}\frac{\left\| e_{k+n} \right\|}{\left\| e_{k} \right\|}\left\| x \right\|\le Ct^n\left\| x \right\|$$
for all $x\in N$ and $n\in \N$. Thus, $B_w$ is generalized hyperbolic with topological direct sum decomposition $X=M\oplus N$. Therefore, by Lemma \ref{LL1}, $B_w$ has the shadowing property.
\end{proof}
\begin{remark}
Observe that, among the conditions defining admissibility, the proof of the implication $(\Rightarrow)$ uses only condition {\rm (C1)} and the fact that $(e_n)_{n\in \Z}\subset X$, whereas the proof of the implication $(\Leftarrow)$ uses only conditions (CB) and {\rm (C2)}.
\end{remark}       
The following corollary is a direct application of Theorem \ref{Teo1} to Example \ref{ex1}.
    \begin{corollary}\label{corter45} Let $\lambda=(\lambda_n)_{n\in\mathbb{Z}}$ and $w:=(w_n)_{n\in\Z}$ be sequences of nonzero scalars that satisfy (\ref{cond}). Let $X\subset \ell^{p}(\lambda,\Z)$ ($1\le p \le \infty$) be a closed subspace that satisfies conditions (CS1)-(CS3). Then $B_w: X \to X$ has the shadowing property if and only if one of the following conditions holds:
        \begin{itemize}
    \item[\rm (I)] $\lim_{n\to \infty}\sup_{k\in \Z}\left( \left| w_{k}\cdots w_{k+n-1} \right|\frac{|\lambda_k|}{|\lambda_{k+n}|} \right)^{\frac{1}{n}}<1;$
    
    \item[\rm (II)] $\lim_{n\to \infty}\inf_{k\in \Z}\left( \left| w_{k}\cdots w_{k+n-1} \right|\frac{|\lambda_k|}{|\lambda_{k+n}|} \right)^{\frac{1}{n}}>1;$
    
    \item[\rm (III)] $\lim_{n\to \infty}\sup_{k\in \N}\left( \left| w_{-k-1}\cdots w_{-k-n} \right|\frac{|\lambda_{-k-n}|}{|\lambda_{-k}|} \right)^{\frac{1}{n}}<1$ and \\ $\lim_{n\to \infty}\inf_{k\in \N}\left( \left| w_{k}\cdots w_{k+n-1} \right|\frac{|\lambda_k|}{|\lambda_{k+n}|} \right)^{\frac{1}{n}}>1$.
\end{itemize}
    \end{corollary}
    For the next example, given a subset $A\subset \N$, $\text{card}(A)$ denotes the cardinality of $A$. Moreover, we use the following notation
    $$\udens(A):=\limsup_{N \to \infty}\frac{\text{card}(\left\{ 1, \cdots, N \right\}\cap A)}{N}.$$
\begin{example}
Let $T:X\to X$ be an operator on a Banach space $X$. We say that $T$ satisfies the {\it Distributional Chaos Criterion} (DCC) if there exist sequences $(x_k)_{k\in \N}$ and $(y_k)_{k\in \N}$ in $X\setminus\{0\}$ such that:
\begin{itemize}
    \item[(DC1)] There exists $A\subset \N$ with $\udens(A)=1$ such that $\lim_{n\in A}T^n( x_k)=0$ for all $k\in\N$;
    \item[(DC2)] $y_k\in \overline{\operatorname{span}\{x_n:n\in\N\}}$, $y_k\to 0$, and there exist $\varepsilon>0$ and an increasing sequence $(N_k)$ in $\N$ such that
    $$\operatorname{card}\left\{1\leq j\leq N_k:\|T^j( y_k)\|>\varepsilon\right\}\geq \varepsilon N_k$$
    for all $k\in\N$.
\end{itemize}
It was proved in \cite[Theorem 12]{Nilson0} that $T$ is distributionally chaotic if and only if $T$ satisfies the DCC. Below, we provide two examples of weighted shifts that completely distinguish the shadowing property from distributional chaos.
   \begin{itemize}
       \item[(a)] Let $\lambda=(\lambda_j)_{j\in \Z}$ be a sequence of scalars given by
       $$\left( \lambda_{j} \right)_{j \in \Z}:=(\underset{j \le0}{\underbrace{\cdots,1,1,1}},1,B_{1},1,1,2,B_{2},2,1, \cdots, 1,2,\cdots,n,B_{n},n,\cdots,2,1,\cdots),$$
		where 
		$$B_{n}:=(\underset{10^n\text{ times}}{\underbrace{n+1,n+1, \cdots, n+1}} ) \quad \text{for all }n\in \N.$$
        We denote by $I_n$ the set of indices corresponding to the block $B_n$ for $n\in \N$. For example, $I_1=\left\{ 2,3,\cdots,11 \right\}$. Consider the bilateral weighted backward shift $B_w: \ell^{\infty}(\lambda,\Z)\to \ell^{\infty}(\lambda,\Z)$ with weights $w:=(w_n)_{n\in \Z}$ given by $w_{-n}=\frac{1}{2}$ if $n\in \N$, and $w_n=1$ if $n\in \N_0$. Note that 
       $$\inf_{k\in\N}\left( \left| w_{k}\cdots w_{k+n-1} \right|\frac{|\lambda_k|}{|\lambda_{k+n}|} \right)^{\frac{1}{n}}=\frac{1}{(n+1)^{\frac{1}{n}}}\quad\text{and}\quad\sup_{k\in\N}\left( \left| w_{k}\cdots w_{k+n-1} \right|\frac{|\lambda_k|}{|\lambda_{k+n}|} \right)^{\frac{1}{n}}=(n+1)^{\frac{1}{n}},$$
       for each $n\in \N$. Then, by Corollary \ref{corter45}, $B_w$ does not have the shadowing property. On the other hand, taking $(x_k)_{k\in \N}:=(e_k)_{k\in \N}$ and $A=\N$, straightforward computations show that (DC1) follows.  Note that if $j= 2\sum_{\ell=1}^{N}\ell +\sum_{\ell=1}^{N}10^\ell$ for some $N \in \N$, then $\lambda_{j}=1$. To prove condition (DC2), for each $k \in \N$ we take $N_k \in \N$ where $N_k:=2\sum_{\ell=1}^{n_k}\ell +\sum_{\ell=1}^{n_k}10^\ell$, for some $n_k \in \N$ sufficiently large such that 
		\begin{equation}\label{eq3w3w}
			\frac{\sum_{\ell=k}^{n_k}10^\ell}{N_k}> 1-\frac{1}{2k},
		\end{equation}
		 and $y_k:=\frac{e_{N_k}}{k}$. Note that $y_k\to 0$. Now, observe that 
		\begin{equation}\label{eqq3w3w}
			\left\| B_{w}^{j}(y_{k}) \right\|=\frac{\left\| w_{N_k-j}\cdots w_{N_k-1} e_{N_k-j} \right\|}{k}= \frac{\left\| e_{N_k-j} \right\|}{k}> 1\quad \forall j \in \N\text{ s.t. }N_k-j\in \bigcup_{i=k}^{n_k}I_{i}.
		\end{equation}
		Thus, by (\ref{eq3w3w}) and (\ref{eqq3w3w}) we obtain
		$$\text{card}\left\{ 1\le j\le N_k: \left\| B_{w}^{j}(y_{k}) \right\|> \frac{1}{2} \right\}\ge \sum_{\ell=k}^{n_k}10^\ell> \left( 1-\frac{1}{2k} \right)N_k\ge \frac{1}{2}N_k.$$
		Then condition (DC2) is satisfied with $\varepsilon=\frac{1}{2}$. Therefore, $B_w$ is distributionally chaotic.
       \item[(b)] Consider the bilateral weighted backward shift $B_w: c(\Z)\to c(\Z)$ with weights $w:=(w_n)_{n\in \Z}$ given by $w_n=2$ for every $n\in \Z$. It is easy to see that $B_w$ satisfies item (II) of Corollary \ref{corter45}. Therefore, $B_w$ has the shadowing property. On the other hand, $B_w$ does not satisfy the condition (DC1) for any choice of the sequence $(x_k)_{k\in \N}$ and the set $A\subset \N$. Indeed, for all $x \in c(\Z)\setminus \{0\}$ and $n \in \N$ we have $\left\| B_{w}^n(x) \right\|=2^n\left\| x \right\|$. Thus, $B_w$ is not distributionally chaotic.
   \end{itemize}
   \end{example}
    \section{The case of weighted composition operators on admissible spaces}
In this section, we establish a complete characterization of the shadowing property for weighted composition operators on admissible spaces, which is stated in Theorem~\ref{MainTheorem}. Before proving this result, we establish in Theorem~\ref{theo1} sufficient conditions for the shadowing property in a more general setting.
    \begin{theorem}\label{theo1}
         Let $(\Omega,f,w)$ be an invertible system. Let $\left( X, \left\| \cdot  \right\| \right)$ be a Banach space of continuous functions $\varphi : \Omega \to \K$ such that the following conditions hold:
         \begin{itemize}
             \item[\rm (A1)] The weighted composition operator $C_{w,f}:X \to X$ is an invertible operator;
             \item[\rm (A2)] $\left\| (C_{w,f})^n(\varphi) \right\|\le \left\| w^{[n]} \right\|_{\infty,\,f^{-n}(\operatorname{supp}(\varphi))}\left\| \varphi \right\|$ for all $\varphi \in X$ and $n\in \Z$;
             \item[\rm (A3)] There exist functions $\rho_-,\rho_+: \Omega \to \K$ that satisfy $\rho_{\pm} \varphi \in X$ and $\left\| \rho_{\pm}\varphi \right\|\le \left\| \varphi \right\|$ for all $\varphi\in X$, and $\rho_- + \rho_+\equiv 1$.
         \end{itemize}
         If one of the following conditions holds,
         \begin{itemize}
         \item[\rm (A)] $\lim_{n \to \infty} \left\| w^{[n]} \right\|_{\infty}^{\frac{1}{n}}<1;$ 
        \item[\rm (B)] $\lim_{n \to \infty} \left\| w^{[-n]} \right\|_{\infty}^{\frac{1}{n}} <1;$
        \item[\rm (C)] $\lim_{n \to \infty} \left\| w^{[n]} \right\|_{\infty,\,f^{-n}(\operatorname{supp}(\rho_-))}^{\frac{1}{n}}<1$ and $\lim_{n \to \infty} \left\| w^{[-n]} \right\|_{\infty,\,f^{n}(\operatorname{supp}(\rho_+))}^{\frac{1}{n}}<1,$
        \end{itemize}
         then $C_{w,f}$ has the shadowing property.
    \end{theorem}
    \begin{proof}
        Suppose that item (A) holds. Then, there exist $C>0$ and $0 < t <1$ such that $\left\| w^{[n]} \right\|_{\infty} \le Ct^n$ for all $n\in \N_0$. Hence, by condition (A2), we have 
        $$ \left\| C_{w,f}^n(\varphi) \right\|\le \left\| w^{[n]} \right\|_{\infty}\left\| \varphi \right\| \le Ct^n \left\| \varphi \right\| \quad \text{for all } \varphi\in X\text{ and } n\in\N_0.$$
       Thus, $C_{w,f}$ is generalized hyperbolic with topological direct sum decomposition $X=X\oplus \left\{ 0 \right\}$. Therefore, by Lemma \ref{LL1}, $C_{w,f}$ has the shadowing property. Assuming now that (B) holds, and using the same arguments as before, we obtain that $C_{w,f}$ is generalized hyperbolic with topological direct sum decomposition $X=\left\{ 0 \right\}\oplus X$. Therefore, once again by Lemma \ref{LL1}, we obtain that $C_{w,f}$ has the shadowing property.

       Now, suppose that (C) holds. Hence, there exist $C>1$ and $0 < t <1$ such that 
\begin{equation}\label{gh1}
   \left\| w^{[n]} \right\|_{\infty,\,f^{-n}(\operatorname{supp}(\rho_-))} \le Ct^n\quad \text{and} \quad \left\| w^{[-n]} \right\|_{\infty,\,f^{n}(\operatorname{supp}(\rho_+))} \le Ct^n\quad\text{for all }n \in \N_0.
\end{equation}
 By (\ref{gh1}) and conditions (A2) and (A3), we have 
\begin{equation}\label{gh2}
    \left\| C_{w,f}^n(\varphi\rho_{-}) \right\|\le\left\|w^{[n]} \right\|_{\infty,\,f^{-n}(\operatorname{supp}(\varphi\rho_{-}))}\left\|\varphi  \rho_{-}\right\|\le\left\|w^{[n]} \right\|_{\infty,\,f^{-n}(\operatorname{supp}( \rho_{-}))}\left\|\varphi \right\|\le Ct^{n}\left\|\varphi  \right\|
\end{equation}
for all $\varphi\in X$ and $n\in \N_0$. Analogously, we obtain
\begin{equation}\label{gh3}
    \left\| (C_{w,f}^{-1})^n(\varphi\rho_{+}) \right\|\le Ct^{n}\left\|\varphi  \right\|\quad \text{ for all }\varphi \in X \text{ and } n\in \N_0.
    \end{equation}
    Fix $\eps >0$. Put $\delta:= \frac{\eps(1-t)}{4C}$ and let $(\varphi_n)_{n \in \Z}$ be a $\delta$-pseudotrajectory. Define, for $n \in\Z$, $\psi_n := \varphi_{n+1}- C_{w,f}(\varphi_n) $. Moreover, for each $n\in \Z$, define
$$\phi_n^{(1)} = \sum_{k=0}^{\infty} C_{w,f}^k (\psi_{n-k-1} \rho_-) \quad \text{and} \quad  \phi_n^{(2)} = -\sum_{k=1}^{\infty} (C_{w,f}^{-1})^k (\psi_{n+k-1} \rho_+).$$
Observe that the above series converge by (\ref{gh2}) and (\ref{gh3}). Put $\phi_n:=\phi_{n}^{(1)}+\phi_{n}^{(2)}$ for each $n\in \Z$. Note that 
\begin{align*}
C_{w,f}(\phi_n)+ \psi_n &= \sum_{k=0}^{\infty} C_{w,f}^{k+1} (\psi_{n-k-1} \rho_-) -\sum_{k=1}^{\infty} (C_{w,f}^{-1})^{k-1} (\psi_{n+k-1} \rho_+) +  \psi_{n}\rho_-+ \psi_{n}\rho_+\\ 
& =\sum_{k=0}^{\infty} C_{w,f}^{k} (\psi_{n-k} \rho_-)-\sum_{k=1}^{\infty} (C_{w,f}^{-1})^{k} (\psi_{n+k} \rho_+)  =\phi_{n+1},
\end{align*}
for all $n\in\Z$. Thus, $\varphi_{n+1}-\phi_{n+1}=C_{w,f}(\varphi_{n}-\phi_{n})$ and, hence, $\varphi_{n}-\phi_{n}=C_{w,f}^n(\varphi_{0}-\phi_{0})$, for all $n \in \Z$. Since $\left\|\psi_n \right\|\le \delta$ for all $n\in\Z$, we obtain by (\ref{gh2}) and (\ref{gh3}) that
$$\left\| \phi_n^{(i)} \right\| \le \sum_{k=0}^{\infty} C\delta t^k=\frac{C\delta}{1-t} \quad \text{for all }i\in\left\{ 1,2 \right\}\text{ and }n\in \Z.$$
Thus, 
$$\left\| \varphi_n - C_{w,f}^n(\varphi_0 - \phi_0) \right\|=\left\| \phi_n \right\|\le \frac{2C\delta}{1-t}=\frac{\eps}{2}<\eps\quad \text{for each }n\in \Z.$$
Therefore, $C_{w,f}$ has the shadowing property.
    \end{proof}
    \begin{remark}\label{rmk1}
        From the proof, we see that items (A) and (B) imply that $C_{w,f}$ is generalized hyperbolic with topological direct sum decompositions $X=X\oplus\{0\}$ and $X=\{0\}\oplus X$, respectively, and is therefore hyperbolic. On the other hand, in the proof of implication $\text{(C)}\Longrightarrow \text{shadowing}$, we obtain a sum decomposition $X=M+N$ given by 
        $$M:=\left\{ \varphi\rho_-: \varphi \in X \right\}\quad \text{and}\quad N:=\left\{ \varphi\rho_+: \varphi \in X \right\},$$
        that is not, in general, a topological direct sum decomposition satisfying Definition \ref{defgen}. Therefore, in item (C), a weaker condition than generalized hyperbolicity is used to prove the shadowing. In other characterizations, such as those in \cite{BCDFP, BerMe, DAnDarMai21}, this does not occur, since generalized hyperbolicity is used in an essential way to establish the shadowing property. This may provide a possible approach to answering Problem F from \cite{BCDFP} in the negative.
    \end{remark}
    Let $\Omega\subset \C$ be a non-empty, open and relatively compact set. We denote by $H(\Omega)$ the space of all holomorphic functions $\varphi:\Omega\to\mathbb{C}$ and by $A(\Omega)$ the space of all continuous functions $\varphi:\overline{\Omega}\to\mathbb{C}$ such that $\varphi|_{\Omega}\in H(\Omega)$. For the next corollary, we first need the following definition.
    \begin{definition}
        Let $\mathbb{D}:=\{z\in\mathbb{C}:|z|<1\}$ and consider $p\in [1, \infty]$. The {\it Hardy space} $H^p(\mathbb{D})$ is the Banach space defined as follows:
        \begin{itemize}
            \item If $p\in [1, \infty)$, then
            $$H^p(\mathbb{D}):=\left\{\varphi\in H(\mathbb{D}):\sup_{0<r<1}\frac{1}{2\pi}\int_0^{2\pi}|\varphi(re^{i\theta})|^p\,d\theta<\infty\right\},$$
            endowed with the norm 
            $$\|\varphi\|_{p}:=\left(\sup_{0<r<1}\frac{1}{2\pi}\int_0^{2\pi}|\varphi(re^{i\theta})|^p\,d\theta\right)^{1/p} \quad \text{for }\varphi \in H^p(\mathbb{D}).$$
            \item If $p=\infty$, then
            $$H^\infty(\mathbb{D}):=\left\{\varphi\in H(\mathbb{D}):\sup_{z\in\mathbb{D}}|\varphi(z)|<\infty\right\},$$
            endowed with the supremum norm 
            $$\|\varphi\|_{\infty}:=\sup_{z\in\mathbb{D}}|\varphi(z)|\quad \text{for } \varphi \in H^{\infty}(\mathbb{D}).$$
        \end{itemize}
    \end{definition} 
    \begin{corollary}\label{corhard}
 Let $w\in A(\mathbb{D})$ be such that $\inf_{z \in \overline{\mathbb{D}}}|w(z)|>0$. Consider $p\in [1, \infty]$ and define $\omega :=w|_\mathbb{D}$. Then the multiplication operator
 $$M_\omega : \varphi \in H^p(\mathbb{D})\mapsto \varphi\cdot \omega \in H^p(\mathbb{D})$$
 is an invertible operator, and the following assertions are equivalent:
 \begin{itemize}
     \item[\rm (a)] $M_\omega $ has the shadowing property;
     \item[\rm (b)] $|w(z)|\neq 1$ for all $z\in \overline{\mathbb{D}}$.
 \end{itemize}
     \end{corollary}
     \begin{proof}
        Fix $p\in [1, \infty]$. Simple computations show that $M_\omega $ is an invertible operator. We now prove the equivalence.

((b)$\Rightarrow$(a)): Let $\text{Id}: z \in \mathbb{D}\mapsto z \in \mathbb{D}$ be the identity map. We see that $(\mathbb{D}, \text{Id},\omega )$ is an invertible system and $\left( H^p(\mathbb{D}) , \left\| \cdot  \right\|_p\right)$ satisfies conditions (A1)-(A3) of Theorem \ref{theo1}. Indeed, since the weighted composition operator $C_{\omega ,\text{Id}}$ coincides with the multiplication operator $M_\omega $, which is an invertible operator, condition (A1) is satisfied. Simple calculations show that condition (A2) holds. For condition (A3), it suffices to take $\rho_-\equiv 1$ and $\rho_+\equiv 0$. Therefore, to establish (a), it suffices to show that either (A) or (B) of Theorem \ref{theo1} holds. Since $|w(z)|\neq 1$ for all $z\in \overline{\mathbb{D}}$, by the connectedness and compactness of $\overline{\mathbb{D}}$ we obtain that there exists $\lambda\in(0,1)$ such that either $|w(z)|<1-\lambda$ for every $z\in\overline{\mathbb{D}}$, or $|w(z)|>1+\lambda$ for every $z\in\overline{\mathbb{D}}$. Thus,
$$\lim_{n\to \infty}\left\| \omega ^{[n]} \right\|_{\infty}^{\frac{1}{n}}=\lim_{n\to \infty}\left\| \omega ^{n} \right\|_{\infty}^{\frac{1}{n}}\le 1-\lambda<1\quad \text{or}\quad\lim_{n\to \infty}\left\| \omega ^{[-n]} \right\|_{\infty}^{\frac{1}{n}}=\lim_{n\to \infty}\left\| \frac{1}{\omega ^{n}} \right\|_{\infty}^{\frac{1}{n}}\le \frac{1}{1+\lambda}<1.$$
Therefore, by Theorem \ref{theo1}, we conclude (a).

         ((a)$\Rightarrow$(b)): Suppose that there exists $z_0\in \mathbb{D}$ such that $|w(z_0)|=1$. It is well known that the functional $F_{z}: \varphi \in H^p(\mathbb{D})\mapsto \varphi(z) \in \mathbb{C}$ is continuous for all $z\in \D$. Thus, there is $C>0$ such that 
         \begin{equation}\label{hardy1}
             \left| \varphi(z_0) \right|\le C \left\| \varphi \right\|_p \quad \text{for all }\varphi\in H^p(\D).
         \end{equation}
         Take $\varepsilon=1$ and let $\delta>0$ be associated with this $\varepsilon$ according to the definition of shadowing. Consider the $\delta$-pseudotrajectory $(\varphi_n)_{n\in \Z}$ given by
         $$\begin{cases}
    \varphi_{-n}\equiv 0, & \text{for }n\in \mathbb{N}_0;  \\
    \varphi_{n}=\omega \cdot \varphi_{n-1}+w(z_0)^{n-1}\delta, & \text{for }n\in \mathbb{N}.
\end{cases}$$
Hence, by the shadowing property, there exists $\psi\in H^p(\mathbb{D})$ such that
\begin{equation}\label{hardy2}
    \left\| \varphi_n - M_\omega ^n(\psi) \right\|_p<1\quad \text{for all } n\in \Z.
\end{equation}
By (\ref{hardy1}) and (\ref{hardy2}), we have 
         \begin{align*}
|\varphi_n(z_0)|&\le|\varphi_n(z_0)-M_\omega ^n(\psi)(z_0)|+|M_\omega ^n(\psi)(z_0)-\varphi_0(z_0)|\\
&=|\varphi_n(z_0)-M_\omega ^n(\psi)(z_0)|+|\varphi_0(z_0)-\psi(z_0)|\\
&\le C\left\| \varphi_n - M_\omega ^n(\psi) \right\|_p + C\left\| \varphi_0 - \psi \right\|_p\\
&\le 2C,
\end{align*}
for all $n\in \N$. On the other hand, take $N\in \N$ such that $N\delta> 2C$. Thus, 
$$|\varphi_N(z_0)|=\left|N\delta w^{N-1}(z_0) \right|>2C,$$
which is a contradiction. Then $|w(z)|\ne 1$ for all $z\in \D$. Now, suppose that there exists $\alpha \in \T$ with $|w(\alpha)|=1$. Then there is a sequence $(z_n)_{n\in \N}\subset \D$ such that $z_n\to \alpha$. Consider the same $\delta>0$ of the first part of this proof and take $j_0\in \N$ satisfying $ \left| w(\alpha) - w(z_{j_0}) \right|<\frac{\delta}{4}.$ Define $F:=\frac{F_{z_{j_0}}}{\left\| F_{z_{j_0}} \right\|}$. Thus, there exists $\phi\in H^p(\D)$ with $\left\| \phi \right\|_p\le 2$ and $F(\phi)=1$. Note that 
\begin{equation}\label{hardy3}
             \left| F(\varphi) \right|\le  \left\| \varphi \right\|_p \quad \text{for all }\varphi\in H^p(\D).
         \end{equation}
Consider the $\delta$-pseudotrajectory $(\phi_n)_{n\in \Z}$ given by
\[
\begin{cases}
\phi_0:=2\phi,\\
\phi_n:=M_\omega(\phi_{n-1})+2w(\alpha)^{n-1}\big(w(\alpha)-w(z_{j_0})\big)\phi, & \text{for } n\in\mathbb{N},\\
\phi_{-n}:=M_\omega^{-1}\left(\phi_{-n+1}-2w(\alpha)^{-n}\big(w(\alpha)-w(z_{j_0})\big)\phi\right), & \text{for } n\in\mathbb{N}.
\end{cases}
\]
Hence, by the shadowing property, there exists $\chi\in H^p(\mathbb{D})$ such that
\begin{equation}\label{hardy4}
    \left\| \phi_n - M_\omega ^n(\chi) \right\|_p<1\quad \text{for all } n\in \Z.
\end{equation}
Using the fact that $F(\phi_n)=2w(\alpha)^{n}$ for all $n\in \Z$, by (\ref{hardy3}) and (\ref{hardy4}) we obtain that 
$$a_n:=\left| 2w^{n}(\alpha)- w^{n}(z_{j_0})F(\chi) \right|<1\quad \text{for all }n\in\Z.$$
By the first part of this proof, we know that $ |w(z_{j_0})|\ne 1$. If $ |w(z_{j_0})|< 1$, then $a_n\to 2$, which is a contradiction. If $ |w(z_{j_0})|>1$, then $a_{-n}\to 2$, which is another contradiction. Hence, $|w(\alpha)|\ne 1$ and, therefore, we conclude (b).
       \end{proof}
       \begin{remark}
       Since items (A) and (B) from Theorem \ref{theo1} imply that $C_{w,f}$ is hyperbolic, as observed in Remark \ref{rmk1}, by the proof of Corollary \ref{corhard} we actually establish the following equivalence:
$$\text{(a)}\Longleftrightarrow \text{(b)}\Longleftrightarrow M_{\omega}\text{ is hyperbolic}.$$
               \end{remark}
       When $p=\infty$, the previous corollary can be generalized to the spaces $H^{\infty}(\Omega)$, where $\Omega\subset\mathbb{C}$ is a non-empty, open, and relatively compact set, and $H^{\infty}(\Omega)$ denotes the space of all bounded holomorphic functions $\varphi:\Omega\to\mathbb{C}$. It is a Banach space when endowed with the supremum norm.
       \begin{corollary}\label{corHin}
           Let $\Omega\subset \C$ be a non-empty, open and relatively compact set. Let $w\in A(\Omega)$ be such that $\inf_{z \in \overline{\Omega}}|w(z)|>0$. Define $\omega :=w|_\Omega$. Then the multiplication operator
 $$M_\omega : \varphi \in H^{\infty}(\Omega)\mapsto \varphi\cdot \omega \in H^{\infty}(\Omega)$$
 is an invertible operator, and the following assertions are equivalent:
 \begin{itemize}
     \item[\rm (a)] $M_\omega $ has the shadowing property;
     \item[\rm (b)] $|w(z)|\neq 1$ for all $z\in \overline{\Omega}$.
 \end{itemize}
 If, in addition, $\Omega$ is connected, then $M_\omega$ has the shadowing property if and only if it is hyperbolic.
       \end{corollary}
       \begin{proof}
           Simple computations show that $M_\omega $ is an invertible operator. We now prove the equivalence.

         ((a)$\Rightarrow$(b)): This implication follows by using the same arguments as in the proof of the implication ((a)$\Rightarrow$(b)) in Corollary~\ref{corhard}.

         ((b)$\Rightarrow$(a)): By hypothesis, we have that $\Omega=A\sqcup B$, where $A=\left\{ z\in \Omega : |w(z)|<1 \right\}$ and $B=\left\{ z\in \Omega : |w(z)|>1 \right\}$. It is easy to see that $(\Omega, \text{Id},\omega )$ is an invertible system. Moreover, taking 
$$\rho_-(z):=\begin{cases}
1, & \text{if } z\in A,\\
0, & \text{if } z\in B,
\end{cases}\quad \text{and} \quad \rho_+(z):=\begin{cases}
0, & \text{if } z\in A,\\
1, & \text{if } z\in B
\end{cases}$$         
         we see that $\left( H^{\infty}(\Omega) , \left\| \cdot  \right\|_{\infty} \right)$ satisfies conditions (A1)-(A3) of Theorem \ref{theo1}. Since $|w(z)|\neq 1$ for all $z\in \overline{\Omega}$, by the compactness of $\overline{\Omega}$ we obtain that there exists $\lambda\in(0,1)$ such that $|w(z)|<1-\lambda$ for every $z\in A$, and $|w(z)|>1+\lambda$ for every $z\in B$. Thus,
$$\lim_{n\to \infty}\left\| \omega ^{[n]} \right\|_{\infty,\,\operatorname{Id}^{-n}(\operatorname{supp}(\rho_-))}^{\frac{1}{n}}=\lim_{n\to \infty}\left\| \omega ^{n} \right\|_{\infty,\,A}^{\frac{1}{n}}\le 1-\lambda<1$$
and 
$$\lim_{n\to \infty}\left\| \omega ^{[-n]} \right\|_{\infty,\, \operatorname{Id}^n(\operatorname{supp}(\rho_+))}^{\frac{1}{n}}=\lim_{n\to \infty}\left\| \frac{1}{\omega ^{n}} \right\|_{\infty,\, B}^{\frac{1}{n}}\le \frac{1}{1+\lambda}<1.$$
Therefore, by Theorem \ref{theo1}, we conclude (a). For the last assertion, it suffices to observe that the connectedness of $\Omega$ guarantees that either $A=\Omega$ or $B=\Omega$.
       \end{proof}
       \begin{remark}
           Using analogous hypotheses on the weight function and the same arguments as in the proof of Corollary \ref{corHin}, we obtain the same characterization as in Corollary \ref{corHin} for multiplication operators acting on spaces of the form $C_b(\Theta)$, where $\Theta$ is a non-empty, locally compact and relatively compact subspace of a Hausdorff space $\Phi$.
       \end{remark}
           By specializing Theorem \ref{theo1} to the setting of translations, described in Example \ref{exx1}, we obtain the following corollary.
    \begin{corollary}\label{cooor1}
        Consider $X:=C_0(\R)$ or $X:=C_b(\R)$. Let $w:\R\to\R$ be a continuous function that satisfies
        \begin{equation}\label{trans1}
            0 < \inf_{x \in \R} |w(x)| \le \sup_{x \in \R}|w(x)|< \infty.
        \end{equation}
        If one of the following conditions holds:  
        \begin{itemize}
            \item[\rm (a)] $\lim_{n\to \infty}\sup_{x\in \R}|w(x)\cdots w(x+n-1)|^{\frac{1}{n}}<1;$
            \item[\rm (b)] $\lim_{n\to \infty}\inf_{x\in \R}|w(x)\cdots w(x+n-1)|^{\frac{1}{n}}>1;$
            \item[\rm (c)] $\lim_{n\to \infty}\sup_{x\in (- \infty,-n+1]}|w(x)\cdots w(x+n-1)|^{\frac{1}{n}}<1$ and \\ 
            $\lim_{n\to \infty}\inf_{x\in [n, \infty)}|w(x-n)\cdots w(x-1)|^{\frac{1}{n}}>1$,
        \end{itemize}
        then the bilateral weighted translation operator $T_w: X\to X$ has the shadowing property.
    \end{corollary}
    \begin{proof}
        Let $f: x\in \R \to x+1\in \R$. We see that $(\R, f,w)$ is an invertible system and $\left( X , \left\| \cdot  \right\|_\infty\right)$ satisfies conditions (A1)-(A3) of Theorem \ref{theo1}. Indeed, since the weighted composition operator $C_{w,f}$ coincides with the bilateral weighted translation operator $T_w$, which is an invertible operator by (\ref{trans1}), condition (A1) is satisfied. Simple calculations show that condition (A2) holds. For condition (A3), take 
        $$\rho_-(x):=
\begin{cases}
1, & \text{if } x\le 0,\\
1-x, & \text{if } x\in[0,1],\\
0, &\text{if }x>1,
\end{cases}\quad \text{and}\quad \rho_+(x):=
\begin{cases}
0, & \text{if } x\le 0,\\
x, & \text{if } x\in[0,1],\\
1, &\text{if }x>1.
\end{cases}$$
     Note that 
     $$f^{-n}(\operatorname{supp}(\rho_-))= (-\infty, -n+1]\quad \text{and} \quad f^{n}(\operatorname{supp}(\rho_+)) = [n, \infty)$$
     for each $n\in \N_0$. Hence, by Theorem \ref{theo1} and by the fact that 
        $$\left\| w^{[n]} \right\|_{\infty,\, A}= \sup_{x\in A}|w(x)\cdots w(x+n-1)|\quad \text{and}\quad \left\| w^{[-n]} \right\|_{\infty,\, A}= \sup_{x\in A}\frac{1}{|w(x-n)\cdots w(x-1)|}$$
        for all $A\subset \R$ and $n\in\N$, the conclusion follows.
    \end{proof}
    \begin{example}\label{exx5}
      Let $X:=C_0(\R)$ or $X:=C_b(\R)$. Consider the bilateral weighted translation operator $T_w: X\to X $ where 
        \[
w(x)=
\begin{cases}
\dfrac12, & x\le 0,\\[2mm]
\dfrac12+\dfrac32x, & 0\le x\le1,\\[2mm]
2, & x\ge1.
\end{cases}
\]    
Fix $n\in \N$ and take $x \in(-\infty,-n+1]$. Hence, $|w(x)\cdots w(x+n-1)|=\frac{1}{2^n}$. Thus, 
     \begin{equation}\label{eqc1}
         \sup_{x\in (-\infty,-n+1]}|w(x)\cdots w(x+n-1)|^{\frac{1}{n}}=\frac{1}{2}.
     \end{equation}
     Now, consider $x\in [n,\infty)$. Then $\frac{1}{2^n}\le\frac{1}{|w(x-n)\cdots w(x-1)|} \le \frac{1}{2^{n-2}}$. Hence 
     \begin{equation}\label{eqc2}
         2\ge\inf_{x\in [n, \infty)}|w(x-n)\cdots w(x-1)|^{\frac{1}{n}}\ge 2^{\frac{n-2}{n}}.
     \end{equation}
     Thus, by (\ref{eqc1}) and (\ref{eqc2}), we obtain condition (c) of Corollary \ref{cooor1}. Therefore, $T_w$ has the shadowing property.
    \end{example}
    \begin{remark}
In the context of the last example, the sum decomposition $X=M+N$ from Remark \ref{rmk1} given by the spaces 
$$M:=\left\{ \varphi\rho_-: \varphi \in X \right\}\quad \text{and}\quad N:=\left\{ \varphi\rho_+: \varphi \in X \right\},$$
is not direct and, hence, does not satisfy Definition \ref{defgen}. This does not imply that $T_w$ is not generalized hyperbolic. In fact, it is generalized hyperbolic. It suffices to consider the topological direct sum decomposition $X=M_0\oplus N_0$, where
$$M_0:=\left\{\varphi\in X:\varphi(x)=0\text{ for all }x\geq 1\right\}$$ and $$N_0:=\left\{\varphi\in X:\varphi(x)=0\text{ for all }x\leq 0\text{ and }\varphi(x)=x\varphi(1)\text{ for all }x\in[0,1]\right\}.$$ Note that finding such a decomposition is not immediate, unlike in the case of $L^p$-spaces (see \cite{DAnDarMai21}). This shows that Theorem \ref{theo1}, and in particular Corollary \ref{cooor1}, may provide a much simpler way to establish the shadowing property, since it is enough to verify the explicit formulas in items (a)--(c) of Corollary \ref{cooor1}, or, more generally, those in items (A)--(C) of Theorem \ref{theo1}.
    \end{remark}
In the setting of admissible spaces, and under suitable assumptions on the weight and composition functions, we can establish the converse of Theorem \ref{theo1}, as we shall see below. Recall that, given a metric space $(X,d)$, a map $f:X\to X$ is said to be {\it distance-preserving} if $d(f(x),f(y))=d(x,y)$ for all $x,y\in X$.
    \begin{theorem}\label{MainTheorem}
       Let $(\Omega,f,w)$ be an invertible system, where $\Omega$ is a metric space, $w$ is uniformly continuous, and $f$ is distance-preserving. Let $X$ be an admissible space for $C_{w,f}$, generated by $W$. Then the weighted composition operator $C_{w,f}$ on $X$ has the shadowing property if and only if one of the following conditions holds:
        \begin{itemize}
         \item[\rm (A)] $\lim_{n \to \infty} \left\| w^{[n]} \right\|^{\frac{1}{n}}_{\infty}<1;$ 
        \item[\rm (B)] $\lim_{n \to \infty} \left\| w^{[-n]} \right\|^{\frac{1}{n}}_{\infty} <1;$
        \item[\rm (C)] $\lim_{n \to \infty} \left\| w^{[n]} \right\|_{\infty,\,\Omega^{n}_{-}}^{\frac{1}{n}}<1$ and $\lim_{n \to \infty} \left\| w^{[-n]} \right\|_{\infty,\,\Omega^{n}_{+}}^{\frac{1}{n}}<1.$
        \end{itemize}
    \end{theorem}
    \begin{proof}
$(\Rightarrow):$ Suppose that $C_{w,f}$ has the shadowing property. We first prove that, for each $x \in \Omega$, one of the following conditions holds:
\begin{itemize}
    \item[($I_A$)] $\lim_{n \to \infty} \sup_{k \in \mathbb{Z}}  
    \left| w^{[n]}(f^k(x)) \right|^{\frac{1}{n}} < 1;$
    \item[($I_B$)] $\lim_{n \to \infty} \sup_{k \in \mathbb{Z}}  
    \left| w^{[-n]}(f^k(x)) \right|^{\frac{1}{n}} < 1;$
    \item[($I_C$)] $\lim_{n \to \infty} \sup_{k \in \mathbb{N}_0}  
    \left| w^{[n]}(f^{-k-n}(x)) \right|^{\frac{1}{n}} < 1$ and $\lim_{n \to \infty} \sup_{k \in \mathbb{N}}  
    \left| w^{[-n]}(f^{k+n}(x)) \right|^{\frac{1}{n}} < 1.$
\end{itemize}
Fix $x \in \Omega$ and consider the map: 
    $$\begin{matrix}
\Pi: & X & \to & S_x \\
 & \varphi & \mapsto & (\varphi(f^{n}(x)))_{n\in \mathbb{Z}}
\end{matrix}$$
By the definition of $S_x$ we see that $\Pi$ is onto. Moreover, since 
$$\left\| \Pi(\varphi) \right\|_{\infty}=\sup_{n\in\Z}|\varphi(f^{n}(x))|\le \left\| \varphi \right\|_{\infty} \quad \text{for all } \varphi\in X,$$
we have that $\Pi$ is continuous. By Lemma \ref{lead1}, $S_x$ is an admissible Banach sequence space for $B_{w_x}$. Since 
$$\Pi(C_{w,f}(\varphi))=(\varphi(f^{n+1}(x))w(f^n(x)))_{n\in \mathbb{Z}}=B_{w_x}(\Pi(\varphi)) \quad \text{for all }\varphi \in X,$$
by Lemma \ref{LL2} we obtain that $B_{w_x}$ has the shadowing property. Therefore, by applying Theorem \ref{Teo1} to $B_{w_x}:S_x \to S_x$, we obtain that one of the conditions $(I_A)$, $(I_B)$, or $(I_C)$ holds. 

Now, denote by $\mathcal{A}$, $\mathcal{B}$, and $\mathcal{C}$ the sets of points satisfying conditions ($I_A$), ($I_B$) and ($I_C$), respectively. We will prove that
    \begin{equation}\label{prin}
       \Omega=\mathcal{A},\quad \text{or}\quad\Omega=\mathcal{B},\quad\text{or}\quad\Omega=\mathcal{C}.
    \end{equation}
Suppose that $\mathcal{A}\ne \emptyset$ and let $x_0 \in \mathcal{A}$. Then there is $\lambda \in (0,1)$ such that 
$$\lim_{n \to \infty} \sup_{k \in \mathbb{Z}} \left| w^{[n]}(f^k(x_0)) \right|^{\frac{1}{n}}=\lambda.$$ 
Take $n_0 \in \N$ such that 
    \begin{equation}\label{c0}
        \sup_{k \in \mathbb{Z}}  
    \left| w^{[n_0]}(f^k(x_0)) \right|^{\frac{1}{n_0}} < \lambda + \varepsilon,
    \end{equation}
    where $\varepsilon>0$ satisfies $\lambda + 2\varepsilon <1$. Since $f$ is distance-preserving, $w$ is uniformly continuous and
\[
0<\inf_{x\in\Omega}|w(x)|
\leq \sup_{x\in\Omega}|w(x)|<\infty,
\]
it follows that $\left| w^{[n_0]} \right|^{\frac{1}{n_0}}$ is also uniformly continuous. Thus, there exists $\delta>0$ such that 
    \begin{equation} \label{c1}
        \left| \left| w^{[n_0]}(x) \right|^{\frac{1}{n_0}}- \left| w^{[n_0]}(y) \right|^{\frac{1}{n_0}} \right|< \varepsilon \quad  \text{for all } x,y\in \Omega \text{ with } d(x,y)<\delta.
    \end{equation}
    Since $f$ preserves the distance, (\ref{c1}) implies that for every $k\in\mathbb Z$ 
    \begin{equation} \label{c2}
        \left| \left| w^{[n_0]}(f^k(x)) \right|^{\frac{1}{n_0}}- \left| w^{[n_0]}(f^k(y)) \right|^{\frac{1}{n_0}} \right|< \varepsilon \quad  \text{for all } x,y\in \Omega \text{ with } d(x,y)<\delta.
    \end{equation}
    Therefore, if $y \in B(x_0, \delta)$, then, by (\ref{c0}) and (\ref{c2}), we have 
    \begin{equation}\label{ww}
        \sup_{k \in \mathbb{Z}}\left| w^{[mn_0]}(f^k(y)) \right|^{\frac{1}{mn_0}}=\sup_{k \in \mathbb{Z}}\left( \prod_{j=0}^{m-1}\left| w^{[n_0]}(f^{k+jn_0}(y)) \right|^{\frac{1}{n_0}} \right)^{\frac{1}{m}}\le \left( \prod_{j=0}^{m-1}(\lambda + 2\varepsilon)\right)^{\frac{1}{m}} = \lambda + 2\varepsilon
    \end{equation}
for each $m \in \N$. By Lemma \ref{lead1}, we know that the limit $ \lim_{n \to \infty} \sup_{k \in \mathbb{Z}} \left| w^{[n]}(f^k(y)) \right|^{\frac{1}{n}}$ exists for all $y\in\Omega$. Thus, by (\ref{ww}) we obtain 
$$ \lim_{n \to \infty} \sup_{k \in \mathbb{Z}}  
    \left| w^{[n]}(f^k(y)) \right|^{\frac{1}{n}}=  \lim_{n \to \infty} \sup_{k \in \mathbb{Z}}  
    \left| w^{[nn_0]}(f^k(y)) \right|^{\frac{1}{nn_0}}\le \lambda + 2\varepsilon<1$$
for all $y \in B(x_0, \delta)$. Therefore, $B(x_0,\delta)\subset \mathcal{A}$. Since $x_0\in\mathcal{A}$ is arbitrary, we obtain that $\mathcal{A}$ is open.

Now, suppose $\mathcal{B}\ne \emptyset$ and let $ x_0\in \mathcal{B}$. Then there is $\lambda\in(0,1)$ such that 
$$\lim_{n \to \infty} \sup_{k \in \mathbb{Z}} \left| w^{[-n]}(f^k(x_0)) \right|^{\frac{1}{n}}=\lambda.$$
Take $\varepsilon>0$ such that $\lambda+2\varepsilon<1$. As in the previous case, we can find  $\delta> 0$ and $n_0\in \N$, such that 
$$\lim_{n \to \infty} \sup_{k \in \mathbb{Z}}  
    \left| w^{[-n]}(f^k(y)) \right|^{\frac{1}{n}}=  \lim_{n \to \infty} \sup_{k \in \mathbb{Z}}  
    \left| w^{[-nn_0]}(f^k(y)) \right|^{\frac{1}{nn_0}}\le \lambda + 2\varepsilon < 1$$
for all $y \in B(x_0, \delta)$. Therefore, since $x_0\in \mathcal{B}$ is arbitrary, we conclude that $\mathcal{B}$ is open.

Now, suppose that $\mathcal{C}\ne \emptyset$ and let $x_0 \in \mathcal{C}.$ Then there are $\lambda_1, \lambda_2 \in (0,1)$ such that 
$$ \lim_{n \to \infty} \sup_{k \in \mathbb{N}_0}  
    \left| w^{[n]}(f^{-k-n}(x_0)) \right|^{\frac{1}{n}}= \lambda_1\quad \text{and}\quad \lim_{n \to \infty} \sup_{k \in \mathbb{N}}  
    \left| w^{[-n]}(f^{k+n}(x_0)) \right|^{\frac{1}{n}} = \lambda_2.
$$ 
Take $\varepsilon>0$ such that $\lambda_i + 2\varepsilon <1$ for $i=1,2$. Choose $n_0\in \N$ such that 
\begin{equation} \label{d1}
    \sup_{k \in \mathbb{N}_0}  
    \left| w^{[n_0]}(f^{-k-n_0}(x_0)) \right|^{\frac{1}{n_0}}< \lambda_1 +\varepsilon\quad \text{and}\quad  \sup_{k \in \mathbb{N}}  
    \left| w^{[-n_0]}(f^{k+n_0}(x_0)) \right|^{\frac{1}{n_0}} < \lambda_2+ \varepsilon.
\end{equation} 
Arguing as in the first case, we can find $\delta>0$ such that, for all $k\in \Z$,
\begin{equation}\label{dr1}
    \left| \left| w^{[n_0]}(f^{-k-n_0}(x)) \right|^{\frac{1}{n_0}}- \left| w^{[n_0]}(f^{-k-n_0}(y)) \right|^{\frac{1}{n_0}} \right|< \varepsilon\quad \text{for all } x,y\in \Omega \text{ with } d(x,y)<\delta
\end{equation}
and
\begin{equation}\label{dr2}
    \left| \left| w^{[-n_0]}(f^{k+n_0}(x)) \right|^{\frac{1}{n_0}}- \left| w^{[-n_0]}(f^{k+ n_0}(y)) \right|^{\frac{1}{n_0}} \right|< \varepsilon\quad \text{for all } x,y\in \Omega \text{ with } d(x,y)<\delta.
\end{equation}
If $y\in B(x_0,\delta)$, then, by (\ref{d1}), (\ref{dr1}), and (\ref{dr2}), we obtain
\begin{align*}
\sup_{k \in \mathbb{N}_0}\left| w^{[mn_0]}(f^{-k-mn_0}(y)) \right|^{\frac{1}{mn_0}} & =\sup_{k \in \mathbb{N}_0}\left( \prod_{j=0}^{m-1}\left| w^{[n_0]}(f^{-k-mn_0+jn_0}(y)) \right|^{\frac{1}{n_0}} \right)^{\frac{1}{m}}\\
&\le\left( \prod_{j=0}^{m-1}(\lambda_1 + 2\varepsilon)\right)^{\frac{1}{m}} = \lambda_1 + 2\varepsilon<1,\tag{$\star$}
\end{align*}
and
\begin{align*}
\sup_{k \in \mathbb{N}}\left| w^{[-mn_0]}(f^{k+mn_0}(y)) \right|^{\frac{1}{mn_0}} & =\sup_{k \in \mathbb{N}}\left( \prod_{j=0}^{m-1}\left| w^{[-n_0]}(f^{k+mn_0-jn_0}(y)) \right|^{\frac{1}{n_0}} \right)^{\frac{1}{m}}\\
&\le\left( \prod_{j=0}^{m-1}(\lambda_2 + 2\varepsilon)\right)^{\frac{1}{m}} = \lambda_2 + 2\varepsilon<1\tag{$\star \star$}
\end{align*}
for all $m\in \N$. Thus, by Lemma \ref{lead1}, $(\star)$ and $(\star \star)$, we have 
$$ \lim_{n \to \infty} \sup_{k \in \mathbb{N}_0}  
    \left| w^{[n]}(f^{-k-n}(y)) \right|^{\frac{1}{n}}\le \lambda_1+2\varepsilon<1\quad \text{and}\quad \lim_{n \to \infty} \sup_{k \in \mathbb{N}}  
    \left| w^{[-n]}(f^{k+n}(y)) \right|^{\frac{1}{n}} \le \lambda_2+2\varepsilon<1,$$
    for all $y \in B(x_0, \delta)$. Therefore, since $x_0\in \mathcal{C}$ is arbitrary, $\mathcal{C}$ is open. Since $\mathcal{A}$, $\mathcal{B}$, and $\mathcal{C}$ are open disjoint sets, whose union is $\Omega$, and $W$ is connected, it follows that exactly one of the following assertions holds: 
$$ W \subset \mathcal{A}, \quad \text{or} \quad W \subset \mathcal{B}, \quad \text{or} \quad W \subset \mathcal{C}.$$
It is easy to see that $x \in \mathcal{A}$ if and only if $f^k(x) \in \mathcal{A}$ for all $k\in\Z$. Thus, by condition (AS2), if $W \subset \mathcal{A}$, then $\Omega = \mathcal{A}$. Similarly, we see that if $W \subset \mathcal{B}$, then $\Omega = \mathcal{B}$. Now, suppose that $x\in \mathcal{C}$. Thus, denoting by $M$ the supremum of $|w|$, we obtain 
\begin{align*}
\lim_{n \to \infty} \sup_{k \in \mathbb{N}_0}  
    \left| w^{[n]}(f^{-k-n}(f(x))) \right|^{\frac{1}{n}}&=\lim_{n \to \infty} \sup_{k \in \mathbb{N}_0}  
    \left| w(f^{-k}(x))w^{[n-1]}(f^{-k-(n-1)}(x)) \right|^{\frac{1}{n}} \\
&\le \lim_{n \to \infty} \sup_{k \in \mathbb{N}_0}  
    \left|Mw^{[n-1]}(f^{-k-(n-1)}(x)) \right|^{\frac{1}{n}} \\
& = \lim_{n \to \infty}\left(  \sup_{k \in \mathbb{N}_0}  
    \left|w^{[n-1]}(f^{-k-(n-1)}(x)) \right|^{\frac{1}{n-1}} \right)^{\frac{n-1}{n}}<1 
\end{align*}
and 
$$\lim_{n \to \infty} \sup_{k \in \mathbb{N}}  
    \left| w^{[-n]}(f^{k+n}(f(x))) \right|^{\frac{1}{n}}\le \lim_{n \to \infty} \sup_{k \in \mathbb{N}}  
    \left| w^{[-n]}(f^{k+n}(x)) \right|^{\frac{1}{n}} <1.$$
Thus, \(f(x)\in\mathcal C\). Similarly, one can check that \(f^{-1}(x)\in\mathcal C\). Therefore, if \(x\in\mathcal C\), then \(f^j(x)\in\mathcal C\) for all \(j\in\mathbb Z\). Hence, if \(W\subset \mathcal C\), then, by condition (AS2), \(\Omega=\mathcal C\). Thus, we conclude (\ref{prin}). Now, we analyze each case separately:
\begin{itemize}
    \item[1.] \underline{\text{Case }$\Omega = \mathcal{A}$}: We will prove that (A) holds. By (\ref{ww}), we know that for each $x \in \Omega$, there exist $n_x\in \N$, $\delta_x>0$ and $\gamma_x \in (0,1)$ such that
$$ \sup_{k \in \mathbb{Z}}  
    \left| w^{[nn_x]}(f^k(y)) \right|^{\frac{1}{nn_x}}\le \gamma_x<1 \quad  \text{for all }y\in B(x,\delta_x)\text{ and }n\in\N.$$
 By the compactness of $\overline{W}$, there exist $m\in \N$ and $x_1, \cdots, x_m\in \overline{W}$ such that $\overline{W}\subset \bigcup_{i=1}^{m}B(x_i,\delta_{x_i})$. Take $n_0:= \prod_{i=1}^{m}n_{x_i}$ and $\gamma_0:=\max\left\{ \gamma_{x_i}:1\le i\le m \right\}$. Then, for all \(y\in\Omega\), we have
\[
\left|w^{[nn_{0}]}(y)\right|^{\frac{1}{nn_0}}\le \gamma_{0}<1
\qquad \text{for all } n\in\mathbb N.
\]

Thus,
\[
\left\|w^{[nn_{0}]}\right\|_{\infty}^{\frac{1}{nn_0}}\le \gamma_{0}<1
\qquad \text{for all } n\in\mathbb N,
\]
and, therefore, by Lemma \ref{lead2}, condition (A) holds.
\item[2.] \underline{\text{Case }$\Omega = \mathcal{B}$}: Arguing as in Case 1, we obtain that if $\Omega=\mathcal{B}$, then condition (B) holds.
\item[3.] \underline{\text{Case }$\Omega = \mathcal{C}$}: We will prove that condition (C) holds. By $(\star )$ and $(\star \star)$, we know that for each $x \in \Omega$, there exist $n_{x}\in \N$, $\delta_{x}>0$ and $\gamma_{x}\in(0,1)$ such that
$$\max\left\{  \sup_{k \in \mathbb{N}_0}\left| w^{[nn_{x}]}(f^{-k-nn_{x}}(y)) \right|^{\frac{1}{nn_{x}}},\sup_{k \in \mathbb{N}}\left| w^{[-nn_x]}(f^{k+nn_x}(y)) \right|^{\frac{1}{nn_{x}}} \right\}\le \gamma_{x} <1$$
 for all $y \in B(x, \delta_x)$ and $n\in \N$. By the compactness of $\overline{W}$, there exist $m\in \N$ and $x_1, \cdots, x_m\in \overline{W}$ such that $\overline{W}\subset\bigcup_{i=1}^{m}B(x_i,\delta_{x_i})$. Take 
 $$n_0:= \prod_{i=1}^{m}n_{x_i}\quad  \text{and}\quad\gamma_0:=\max\left\{ \gamma_{x_i}:1\le i\le m \right\}.$$
 Then we have 
 $$ \left| w^{[nn_{0}]}(y) \right|^{\frac{1}{nn_{0}}}\le \gamma_{0}<1 \quad \text{for all }n \in \N\text{ and }y\in \Omega_{-}^{nn_0},$$
 and 
 $$ \left| w^{[-nn_0]}(y) \right|^{\frac{1}{nn_{0}}}\le \gamma_{0}<1 \quad \text{for all }n \in \N\text{ and }y\in \Omega_{+}^{nn_0+1}.$$
 Thus, for a suitable constant $C>0$, we have
$$\left\| w^{[nn_{0}]} \right\|_{\infty, \,\Omega^{nn_0}_{-}}^{\frac{1}{nn_{0}}}\le \gamma_{0}\quad \text{and}\quad \left\| w^{[-nn_{0}-1]} \right\|_{\infty, \,\Omega^{nn_0+1}_{+}}^{\frac{1}{nn_{0}+1}}\le  C^{\frac{1}{nn_0+1}}\gamma_{0}^{\frac{nn_0}{nn_0+1}} \quad \text{for all }n \in \N$$
and, therefore, by Lemma \ref{lead2}, condition (C) holds. 
\end{itemize}
$(\Leftarrow):$ This implication follows directly from Theorem \ref{theo1}.
\end{proof}
The following corollary follows from Lemma \ref{lead3} and Theorem \ref{MainTheorem}.
    \begin{corollary}\label{csb1}
        Let $(\Omega,f,w)$ be an invertible system as in Theorem \ref{MainTheorem}, and let $X:= C_0(\Omega)$ or $X:= C_b(\Omega)$. Let $C_{w,f}$ be a weighted composition operator that is an invertible operator on $X$. Suppose that conditions (H1) and (H2) hold. Then $C_{w,f}$ has the shadowing property if and only if one of the following conditions holds:
        \begin{itemize}
         \item[\rm (A)] $\lim_{n \to \infty} \left\| w^{[n]} \right\|_{\infty}^{\frac{1}{n}}<1;$ 
        \item[\rm (B)] $\lim_{n \to \infty} \left\| w^{[-n]} \right\|_{\infty}^{\frac{1}{n}} <1;$
        \item[\rm (C)] $\lim_{n \to \infty} \left\| w^{[n]} \right\|_{\infty,\,\Omega^{n}_{-}}^{\frac{1}{n}}<1$ and $\lim_{n \to \infty} \left\| w^{[-n]} \right\|_{\infty,\,\Omega^{n}_{+}}^{\frac{1}{n}}<1.$
        \end{itemize}
    \end{corollary}
    By specializing Corollary \ref{csb1} to the setting of translations, described in Example \ref{exx1}, we obtain the following corollary.
    \begin{corollary}\label{cc0r}
Consider $X:= C_0(\R)$ or $X:= C_b(\R)$. Let $w:\R\to\R$ be a uniformly continuous function that satisfies 
$$0 < \inf_{x \in \R} |w(x)| \le \sup_{x \in \R}|w(x)|< \infty.$$ 
Then the bilateral weighted translation operator $T_w: X\to X$ has the shadowing property if and only if one of the following conditions holds:  
        \begin{itemize}
            \item[\rm (a)] $\lim_{n\to \infty}\sup_{x\in \R}|w(x)\cdots w(x+n-1)|^{\frac{1}{n}}<1;$
            \item[\rm (b)] $\lim_{n\to \infty}\inf_{x\in \R}|w(x)\cdots w(x+n-1)|^{\frac{1}{n}}>1;$
            \item[\rm (c)] $\lim_{n\to \infty}\sup_{x\in (- \infty,-n+1]}|w(x)\cdots w(x+n-1)|^{\frac{1}{n}}<1$ and \\ 
            $\lim_{n\to \infty}\inf_{x\in [n, \infty)}|w(x-n)\cdots w(x-1)|^{\frac{1}{n}}>1$.
        \end{itemize}
    \end{corollary}
   \begin{example}
       In the context of Theorem \ref{MainTheorem}, taking $w\equiv 1$ and $f$ to be any distance-preserving map, we obtain that the composition operator $C_f$ never has the shadowing property. In particular, taking $X:=C_0(\R)$ or $X:=C_b(\R)$ we obtain that the translation
\[
\varphi\in X\mapsto \varphi(\cdot+1)\in X
\]
does not have the shadowing property.
   \end{example}
   \begin{example}
Let $X:= C_0(\R)$ or $X:= C_b(\R)$. Consider $w:\R\to\R$ defined by $w(x)=\sin(2\pi x)+2$. For every $k\in \Z$ and $n\in \N$, we have  
$$\left( \sin\left( 2\pi \left( \frac{3}{4}+k \right) \right) +2\right)\cdots \left( \sin\left( 2\pi \left( \frac{3}{4}+k +n-1\right) \right)+2 \right)=1.$$
Therefore, by Corollary \ref{cc0r}, $T_w$ does not have the shadowing property.
   \end{example}


	\smallskip
	
	{\footnotesize

\bigskip\noindent
{\sc Jo\~ao V. A. Pinto}

\smallskip\noindent
Departamento de Matem\'atica Aplicada, Instituto de Matem\'atica, Universidade Federal do Rio de Janeiro, 
Caixa Postal 68530, RJ 21941-909, Brazil.\\
\textit{ e-mail address}: joao.pinto@ufu.br

}

\end{document}